\documentclass[12pt]{amsart}
\usepackage[letterpaper,top=3cm,bottom=3cm,left=3.25cm,right=3.25cm,marginparwidth=1.75cm]{geometry}
\usepackage[english]{babel}
\usepackage{graphicx}
\usepackage{xurl}
\graphicspath{{./}}

\usepackage[T1]{fontenc}
\usepackage[utf8]{inputenc}
\usepackage{lmodern}

\usepackage{algorithm}
\usepackage{algpseudocode}

\usepackage{xcolor}

\algrenewcommand\algorithmicrequire{\textcolor{blue}{\textbf{Input:}}}
\algrenewcommand\algorithmicensure{\textcolor{blue}{\textbf{Output:}}}
\algrenewcommand\algorithmiccomment[1]{%
  \hfill\textcolor{red}{$\triangleright$}~{\footnotesize #1}}
\newcommand{\algrule}{\vspace{1pt}\hrule height 0.7pt\vspace{4pt}}
\usepackage{tikz}
\usetikzlibrary{arrows.meta,positioning}
\usepackage{amsmath,amsthm,amssymb,amsfonts,amsopn}
\numberwithin{equation}{section}
\usepackage{booktabs}
\usepackage{array}
\usepackage{float}
\usepackage{placeins}
\usepackage{needspace}
\usepackage{enumitem}
\newtheorem{theoremp}{Theorem}[section]
\newenvironment{theorem}{\vspace{5pt}\begin{theoremp}}{\end{theoremp}\vspace{5pt}}
\newtheorem{informaltheoremp}[theoremp]{Informal Proposition}

\newtheorem{propositionp}[theoremp]{Proposition}
\newenvironment{proposition}{\vspace{5pt}\begin{propositionp}}{\end{propositionp}\vspace{5pt}}
\newtheorem{corollaryp}{Corollary}[section]
\newenvironment{corollary}{\vspace{5pt}\begin{corollaryp}}{\end{corollaryp}\vspace{5pt}}
\newtheorem{lemmap}[theoremp]{Lemma}
\newenvironment{lemma}{\vspace{5pt}\begin{lemmap}}{\end{lemmap}\vspace{5pt}}
\newtheoremstyle{assumptionstyle}
    {12pt plus 3pt minus 2pt}
    {\topsep}
    {\normalfont}
    {}
    {\itshape}
    {.}
    {.5em}
    {}
\theoremstyle{assumptionstyle}
\newtheorem{asu}{Assumption}

\newtheoremstyle{boldremarkstyle}
    {\topsep}
    {\topsep}
    {\normalfont}
    {}
    {\bfseries}
    {.}
    {.5em}
    {}
\theoremstyle{boldremarkstyle}
\newtheorem{remarkp}{Remark}[section]
\newenvironment{remark}{\vspace{5pt}\begin{remarkp}}{\end{remarkp}\vspace{5pt}}

\theoremstyle{definition}
\newtheorem{examplep}[theoremp]{Example}
\newenvironment{example}{\vspace{5pt}\begin{examplep}}{\end{examplep}\vspace{5pt}}

\newcommand{\stepheading}[2]{%
  \par\smallskip\noindent{\normalfont\bfseries Step #1}%
  \if\relax\detokenize{#2}\relax\else{\normalfont\ (#2)}\fi. }
\newenvironment{step1}[1][]{\stepheading{1}{#1}}{\par\smallskip}
\newenvironment{step2}[1][]{\stepheading{2}{#1}}{\par\smallskip}
\newenvironment{step3}[1][]{\stepheading{3}{#1}}{\par\smallskip}
\newenvironment{step4}[1][]{\stepheading{4}{#1}}{\par\smallskip}

\providecommand{\doi}[1]{\href{https://doi.org/#1}{\nolinkurl{https://doi.org/#1}}}

\makeatletter
\def\subsubsection{\@startsection{subsubsection}{3}%
  \z@{.5\linespacing\@plus.7\linespacing}{-.5em}%
  {\normalfont\bfseries}}
\def\paragraph{\@startsection{paragraph}{4}%
  \z@\z@{-\fontdimen2\font}%
  {\normalfont\bfseries\boldmath}}
\makeatother

\usepackage{caption}
\usepackage{subcaption}
\usepackage[colorlinks=true,linkcolor=blue,citecolor=green,urlcolor=green]{hyperref}
\newenvironment{paperfigure}[1][htbp]{\begin{figure}[#1]}{\end{figure}}
\makeatletter
\def\@adminfootnotes{%
  \let\@makefnmark\relax \let\@thefnmark\relax
  \ifx\@empty\@date\else \@footnotetext{\@setdate}\fi
  \ifx\@empty\thankses\else \@footnotetext{%
    \def\par{\let\par\@par}\@setthanks}\fi
  \ifx\@empty\@subjclass\else \@footnotetext{\@setsubjclass}\fi
  \ifx\@empty\@keywords\else \@footnotetext{\@setkeywords}\fi}
\makeatother
\title[Singular parameters and missing limits in neural PDE solvers]{Singular Parameters and Missing Limits in Neural PDE Solvers}
\hypersetup{pdftitle={Singular Parameters and Missing Limits in Neural PDE Solvers}}
\author[Daniel Fern\'andez]{Daniel Fern\'andez\textsuperscript{$\ast$}}
\thanks{\textsuperscript{$\ast$}Chair for Dynamics, Control, Machine Learning, and Numerics (Alexander von Humboldt Professorship), Department of Mathematics, Friedrich-Alexander-Universit\"at Erlangen-N\"urnberg, 91058 Erlangen, Germany.\newline
\emph{Email:} \texttt{daniel.fernandez@fau.de}.}
\date{}
\begin{document}
\begin{abstract}
Neural solvers for partial differential equations (PDEs) can approach an
accurate solution while their parameters grow without bound. In such cases,
the limiting solution may have no finite representation in the chosen model,
leaving the best loss unattained.
Our analysis connects missing limits in deep neural tanh-networks to unbounded
hidden parameters or increasingly redundant neurons. For a class of models
built from translated kernels, we describe the
missing functions and recover them by adding kernel derivatives to the model.
This completion makes the best approximation attainable under standard
assumptions. Numerical studies follow the associated parameter growth and
explore how completion affects PDE optimization.
\end{abstract}
\maketitle
\section{Introduction}\label{sec:NPDE}

Neural PDE solvers can approach a smooth solution that no finite choice of
model parameters represents, while their weights diverge and their neurons
become nearly identical \cite{mahan,petersen}. The best loss can therefore remain
unattained, even with exact arithmetic and ideal optimization
\cite{BertoluzzaBurmanHe2023WAN}. This possibility matters for both
energy-based and residual-based solvers
\cite{EYu18,KarniadakisNRP21,Raissi19,SirignanoSpiliopoulos18}.

This is the \emph{coercivity gap}: control of the computed solution need
not give control of its parameters
\cite{Zuazua2026CoercivityGap,Zuazua2026SIAMNews}. A simple collision
explains how it arises. Two shifted features with large opposing
coefficients can approach a derivative as their centres merge. We show how adding the
derivative as a feature makes the limit available. We call this
\emph{completion}.

To our knowledge, this is the first systematic computational study of the
coercivity gap and the use of completion to address it in neural PDE solvers. Building on the theoretical account
and two-neuron illustration in \cite{Zuazua2026CoercivityGap}, we combine
new analysis with numerical studies of parameter growth, missing limits,
and the effect of completion on optimization.

\subsection{Contributions}
\label{subsec:introduction_contributions}
The main contributions are:

\begin{itemize}[leftmargin=2.5em,labelsep=0.6em]
\item \textbf{Identifying singular representations.}
For deep neural tanh-networks, approaching a missing limit forces hidden parameters
to escape or hidden layers to degenerate
(Theorem~\ref{thm:tanh_forced_collision}). For translated feature architectures,
approaching a missing limit likewise forces the coefficient norm to diverge
and feature centres to collide.

\item \textbf{Recovering missing optimal states.}
We identify the missing states that can prevent attainment by characterizing
every weak limit of a fixed number of translates of a sufficiently smooth
kernel on the real line
(Theorem~\ref{thm:translated_feature_geometry}). Nonattainment
persists under small target perturbations, and we quantify the weight
growth needed to approach the optimal loss
(Theorem~\ref{thm:robust_nonattainment}).

\item \textbf{Completion algorithm.}
Derivative features restore attainment under the stated variational
assumptions (Corollary~\ref{cor:completed_variational_attainment}).
We propose a completion algorithm that replaces nearly colliding features
by derivative blocks
(Algorithm~\ref{alg:basic_completion}).

\item \textbf{Testing what completion changes.}
Numerical studies track collisions and check the predicted coefficient-growth
law. Manufactured PDE comparisons train kernel weights and
compare standard optimization with completion with derivative features.
\end{itemize}

\subsection{Related work}
\label{subsec:introduction_related_work}
We place our results in the context of approximation theory, neural PDE
solvers, and derivative completion.

\noindent\textbf{Approximation and nonattainment.}
Fixed-size neural realization classes can be nonclosed, and nearby
realized functions need not admit nearby parameter representations
\cite{petersen}. Nonclosedness also occurs in Sobolev norms, where
derivative limits can be approached through parameter growth \cite{mahan}.
Best empirical-risk approximations may fail to exist even for shallow
networks \cite{LimMichalekQi2022}. An analogous
obstruction occurs in fixed-rank tensor approximation, where nonclosed
rank constraints can prevent attainment \cite{deSilvaLim2008}. These
results motivate asking which limits are missing from a fixed model and
how fast its coefficients must grow to approach them.

\noindent\textbf{Neural PDE solvers and training.}
The Deep Ritz method minimizes variational energies \cite{EYu18}, while
physics-informed neural networks (PINNs) and the deep Galerkin method (DGM)
use PDE residuals together with boundary or data losses
\cite{Raissi19,SirignanoSpiliopoulos18}. Convergence analyses
\cite{ShinDarbonKarniadakis20,Doum_che_2023} and stability-based
generalization estimates \cite{MishraMolinaro23} relate losses and
sampling to solution accuracy under their respective assumptions.
Training studies identify numerical stiffness and unbalanced gradients
\cite{WangTengPerdikaris21SISC}, as well as different convergence rates
of loss components through neural tangent kernel analysis
\cite{WangYuPerdikaris20NTK}. Studies of neuron condensation describe
the emergence of groups with similar features during training
\cite{condensation3,condensation1,condensation2}. The coercivity gap
\cite{Zuazua2026CoercivityGap,Zuazua2026SIAMNews} connects these
representation phenomena to PDE objectives: states can converge while
parameters diverge. We examine whether the fixed model contains the
limiting state and whether its best loss is attained.

\noindent\textbf{Derivative completion.}
Derivative completion has classical roots in the theory of
\(\gamma\)-polynomials and nonlinear best approximation
\cite{deBoor1969,Braess1986}. Girosi and Poggio explicitly proposed
adding parameter derivatives to radial basis families with movable centres
to recover the best-approximation property
\cite{GirosiPoggio1989Report,GirosiPoggio1990}.
Within this framework, we identify all weak limits of Sobolev
translates with unrestricted centres, controlling both collisions and
escape despite divergent coefficients.
\FloatBarrier

\section{Preliminaries}
\label{sec:boundary_states}

We first show how neural states can converge while their parameters diverge,
then introduce the variational formulation and model classes used in the
analysis.

\subsection{A shallow-\texorpdfstring{\(\tanh\)}{tanh} example}

Two nearly identical neurons can produce an increasingly accurate state
while their output weights diverge. The following example shows how this
collision can prevent attainment of a coercive PDE objective at fixed width.

\label{sec:motiv_example}

Let \(d\ge1\), let \(\Omega\subset\mathbb R^d\) be a bounded Lipschitz domain,
and set
\[
u^\ast(x):=\operatorname{sech}^2(x_1).
\]
Fix \(P\ge2\). For \(0<\varepsilon\le1\), the width-two shallow network
\begin{equation}
\label{eq:tanh_motivating_difference_quotient}
u_\varepsilon(x)
:=\frac{\tanh(x_1+\varepsilon)-\tanh(x_1-\varepsilon)}{2\varepsilon}
\end{equation}
is also a width-\(P\) network after zero-weight padding. Its slopes are fixed
and its biases lie in \([-1,1]\), so its hidden parameters stay in a compact
set. Smooth Taylor expansion on the bounded preactivation range gives, for every fixed
\(m\ge1\),
\begin{equation}
\label{eq:tanh_motivating_strong_limit}
u_\varepsilon\longrightarrow u^\ast
\qquad\text{strongly in }H^m(\Omega).
\end{equation}
Writing \(e_1\) for the first coordinate vector in \(\mathbb R^d\), the two hidden parameters and outer coefficients are
\[
z_{1,\varepsilon}=(e_1,\varepsilon),\qquad
z_{2,\varepsilon}=(e_1,-\varepsilon),\qquad
c_{1,\varepsilon}=\frac1{2\varepsilon},\qquad
c_{2,\varepsilon}=-\frac1{2\varepsilon}.
\]
Thus
\[
|z_{1,\varepsilon}-z_{2,\varepsilon}|=2\varepsilon\to0,
\qquad
\|(c_{1,\varepsilon},c_{2,\varepsilon})\|_2
=\frac1{\sqrt2\,\varepsilon}\to\infty.
\]

Figure~\ref{fig:tanh_singular_limit} illustrates this explicit family:
the hidden parameters collide while the represented states approach the
regular limit. Shading marks the error of the widest pair.
\begin{paperfigure}[!t]
\centering
\includegraphics[width=\linewidth]{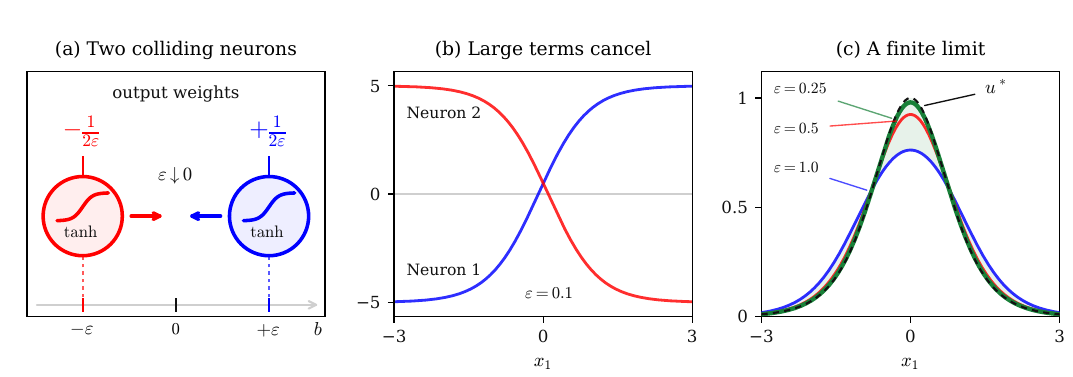}
\caption{Two merging tanh neurons approach a function outside the model.
\textbf{(a)} Weights grow with opposite signs; \textbf{(b)} large contributions
nearly cancel; \textbf{(c)} their sum approaches \(u^\ast\).}
\label{fig:tanh_singular_limit}
\end{paperfigure}

The limit is not represented by any finite shallow-\(\tanh\) network, even if
an output bias is allowed. Indeed, an identity
\[
\operatorname{sech}^2(x_1)
=c_0+\sum_{k=1}^Q c_k\tanh(a_k\cdot x+b_k)
\quad\text{on }\Omega
\]
would first hold pointwise by continuity and then on all of \(\mathbb R^d\)
by real analyticity. Restricting to
\(x=te_1\) would express \(\operatorname{sech}^2t\) as a finite sum of a
constant and terms \(\tanh(\alpha_kt+b_k)\), where \(\alpha_k:=a_k\cdot e_1\).
Every nonconstant term has only
simple poles in its meromorphic continuation, so their finite sum has at most
simple poles. In contrast, \(\operatorname{sech}^2t\) has a double pole at
\(i\pi/2\). The meromorphic identity theorem makes this impossible. Thus, for every
fixed width \(P\ge2\), \(u^\ast\) is a strong \(H^m\) limit of shallow
\(\tanh\) networks but has no such representation.

For a direct elliptic PDE realization, take \(m=1\), let \(\nu\) be the outer unit
normal on \(\partial\Omega\), and write
\[
\partial_\nu u:=\nabla u\cdot\nu.
\]
Set
\[
f:=-\Delta u^\ast+u^\ast,
\qquad
g:=\partial_\nu u^\ast.
\]
Then \(u^\ast\) is the unique minimizer in \(H^1(\Omega)\) of the coercive
Deep Ritz energy
\[
J(u):=\frac12\int_\Omega\bigl(|\nabla u|^2+|u|^2\bigr)\,dx
-\int_\Omega fu\,dx-\int_{\partial\Omega}gu\,dS.
\]
Indeed,
\[
J(u)-J(u^\ast)=\frac12\|u-u^\ast\|_{H^1(\Omega)}^2.
\]
Consequently, restricting \(J\) to this fixed-width shallow-\(\tanh\) class preserves
the state-space infimum but destroys attainment.
Every minimizing parameter sequence therefore escapes, by the
boundary-state principle of \cite{Zuazua2026CoercivityGap}. This example motivates
Theorem~\ref{thm:tanh_forced_collision}, which describes parameter escape
and hidden-layer degeneracy for deep neural tanh-networks of arbitrary fixed depth.

\FloatBarrier


\subsection{Variational and model preliminaries}
\label{subsec:preliminaries}

Let \(H\) be a real Hilbert space, such as \(H^m(\Omega)\) on a bounded
domain or \(H^m(\mathbb R)\), and let \(J:H\to\mathbb R\) be an energy
or residual objective. The variational problem is
\begin{equation}\label{eq:variational_problem}
\min_{u\in H}J(u).
\end{equation}
A neural model is described by a nonempty parameter
set \(\Theta\subseteq\mathbb R^{d_\theta}\) and a continuous realization
map \(T:\Theta\to H\). Its realization class is
\(\mathcal M:=T(\Theta)\). The corresponding parametrized problem is
\begin{equation}\label{eq:parametrized_neural_problem}
\inf_{\theta\in\Theta}J(T(\theta))
=\inf_{u\in\mathcal M}J(u).
\end{equation}
When \(J\) is coercive, meaning \(J(u)\to+\infty\) as
\(\|u\|_H\to\infty\), and weakly lower semicontinuous, the full problem
has a minimizer by the direct method. For the results concerning a unique
full-space minimizer, we use the following assumption:
\begin{asu}\label{asu1}
On a real Hilbert space \(H\), the functional \(J:H\to\mathbb R\) is
continuous, weakly lower semicontinuous, strongly convex, and coercive:
\(J(u)\to+\infty\) as \(\|u\|_H\to+\infty\).
\end{asu}

Under this assumption, the direct method and strong convexity yield a unique
minimizer \(u^\ast\) of \(J\) on \(H\). The neural infimum can fail
to be attained when convergent representable states approach a missing
limit.

\paragraph{Deep neural tanh-networks.}
Let \(\Omega\subset\mathbb R^d\) be bounded, nonempty, and open, and
let \(m\ge1\) be an integer. Fix an architecture \(\mathbf p=(p_0,\ldots,p_L)\), with
\(L\ge1\), \(p_0=d\), and \(p_\ell\ge1\). Its hidden states are
\begin{equation}\label{eq:deep_tanh_model_main}
h^0_{\mathbf z}(x)=x,\qquad
h^\ell_{\mathbf z}(x)=\tanh\!\left(
W^\ell h^{\ell-1}_{\mathbf z}(x)+b^\ell\right),
\quad 1\le\ell\le L,
\end{equation}
where \(W^\ell\in\mathbb R^{p_\ell\times p_{\ell-1}}\),
\(b^\ell\in\mathbb R^{p_\ell}\), \(\tanh\) acts componentwise, and
\(\mathbf z=((W^\ell,b^\ell))_{\ell=1}^L\in\mathbb R^q\) collects
all hidden parameters. With total parameters \(\theta=(c,\mathbf z)\),
the network output is
\[
u_\theta=c_0+\sum_{j=1}^{p_L}c_jh^L_{\mathbf z,j}\big|_\Omega.
\]
If an output bias is omitted, \(c_0=0\). All parameters are real
and unrestricted; there are no input skip connections.

For each neuron let \(a_i^\ell=W^\ell_{i,:}\) be its incoming weight
row and \(z_i^\ell=(a_i^\ell,b_i^\ell)\) its augmented parameter vector.
Define
\begin{equation}\label{eq:deep_tanh_gauge_main}
\Delta_{\mathrm{deep}}(\mathbf z):=
\min\left\{
\min_{\ell,i}\|a_i^\ell\|_2,
\min_{\ell,\ i<j}\|z_i^\ell-z_j^\ell\|_2,
\min_{\ell,\ i<j}\|z_i^\ell+z_j^\ell\|_2
\right\},
\end{equation}
with each minimum over absent pairs equal to \(+\infty\). Thus the degeneracy measure
detects a vanishing incoming row or two neurons in the same layer whose
incoming parameters agree up to sign. Biases enter the pair comparison,
but not the row norm: a zero row makes a neuron constant even with a
nonzero bias.

\paragraph{Translated-feature models.}
Let \(m\ge1\) and \(P\ge1\) be integers, and fix a nonzero,
real-valued kernel \(K\in H^m(\mathbb R)\). Write
\begin{equation}\label{rbfnn}
u_\theta(x)=\sum_{k=1}^P w_kK(x-b_k),\qquad x\in\mathbb R,
\end{equation}
where \(\theta=(w,b)=((w_k,b_k))_{k=1}^P\), and set
\[
\mathcal M_P:=\left\{\sum_{k=1}^P w_kK(\cdot-b_k):
w_k\in\mathbb R,\ b_k\in\mathbb R\right\}.
\]
Zero weights allow fewer than \(P\) features. All weights and centres are
real and unrestricted. Here \(H=H^m(\mathbb R)\),
\(\Theta=\mathbb R^{2P}\), and \(T(\theta)=u_\theta\).
We write \(\mathcal C_P:=\overline{\mathcal M_P}^{\,H^m(\mathbb R)}\)
for the completed class. The completion results below assume the additional
regularity \(K\in H^{m+P-1}(\mathbb R)\).

For the Fourier bounds in Appendix~\ref{subsec:translated_weight_control_appendix},
we use the following assumption in dimension \(d\ge1\). It is not required
for the one-dimensional completion theorem.
\begin{asu}\label{asu:kernel}
Let \(K\in W^{m,1}(\mathbb R^d)\cap H^{m+1}(\mathbb R^d)\) be real-valued
and let its Fourier transform \(\widehat K\) be nowhere zero.
\end{asu}


\section{Main results: the missing boundary and its completion}
\label{sec:main_results}

Missing limits force singular parameters in deep neural tanh-networks. For
the translated-feature class defined above, derivative completion captures every
weak limit and restores variational attainment. In the original class,
nonattainment persists under target perturbations, with a sharp
coefficient-growth law.

\subsection{Missing limits in deep neural tanh-networks force singularity}
\label{subsec:deep_tanh_geometry}
\label{subsec:shallow_tanh_geometry}

For the architecture~\eqref{eq:deep_tanh_model_main}, the degeneracy
measure~\eqref{eq:deep_tanh_gauge_main} gives the following alternative.
\begin{theorem}[Hidden-parameter escape or finite hidden-layer degeneracy]
\label{thm:tanh_forced_collision}
Let \(\Omega\subset\mathbb R^d\) be bounded, nonempty, and open, and
let \(m\ge1\) be an integer. Let \(J:H^m(\Omega)\to\mathbb R\) satisfy
Assumption~\ref{asu1}, and suppose that its minimizer \(u^\ast\) is not representable by
the fixed neural architecture \eqref{eq:deep_tanh_model_main}.
Then any sequence of networks \(u_{\theta_n}\) of this architecture
with parameters \(\theta_n=(c_n,\mathbf z_n)\) satisfying
\[
J(u_{\theta_n})\longrightarrow J(u^\ast)=\inf_{v\in H^m(\Omega)}J(v)
\]
has the following properties:
\begin{enumerate}[label=\textup{(\roman*)},leftmargin=*]
\item The total parameter norm satisfies
\[
\|\theta_n\|_2\longrightarrow\infty.
\]
\item Along every subsequence satisfying \(\sup_k\|\mathbf z_{n_k}\|_2<\infty\),
\begin{equation}\label{eq:tanh_forced_collision_main}
\|c_{n_k}\|_2\to\infty,
\qquad
\Delta_{\mathrm{deep}}(\mathbf z_{n_k})\to0.
\end{equation}
\item Globally,
\begin{equation}\label{eq:deep_tanh_global_alternative_main}
\min\left\{
\frac{1}{1+\|\mathbf z_n\|_2},
\Delta_{\mathrm{deep}}(\mathbf z_n)
\right\}\longrightarrow0.
\end{equation}
\end{enumerate}
\end{theorem}
The energy hypothesis requires a zero approximation gap:
\(\inf_{v\in\mathcal M}J(v)=J(u^\ast)\), equivalently
\(u^\ast\in\overline{\mathcal M}^{\,H^m(\Omega)}\). Minimizing the restricted
objective alone does not imply this equality.
\begin{remark}[Arbitrary convergent sequences]
\label{rem:deep_arbitrary_sequences}
More generally, the same conclusions hold for every sequence of networks
\(u_{\theta_n}\to u\) strongly in \(H^m(\Omega)\) whose limit is not
representable by the fixed architecture \eqref{eq:deep_tanh_model_main}.
This geometric statement is independent of the objective \(J\).
\end{remark}
When hidden parameters remain bounded, output weights diverge while, in
at least one hidden layer, two neurons approach a signed collision or one
neuron becomes constant
(Figure~\ref{fig:forced_singularity}). This condition is necessary, not
sufficient: inactive neurons can make \(\Delta_{\mathrm{deep}}=0\) even when
the represented state belongs to the original class. The result does not
characterize the full closure of deep neural tanh-networks.
Appendix~\ref{app:deep_tanh_proof} proves it using compactness and
identifiability of deep neural tanh-networks \cite{VlacicBolcskei2021}.
\begin{paperfigure}[!t]
\centering
\includegraphics[width=\linewidth]{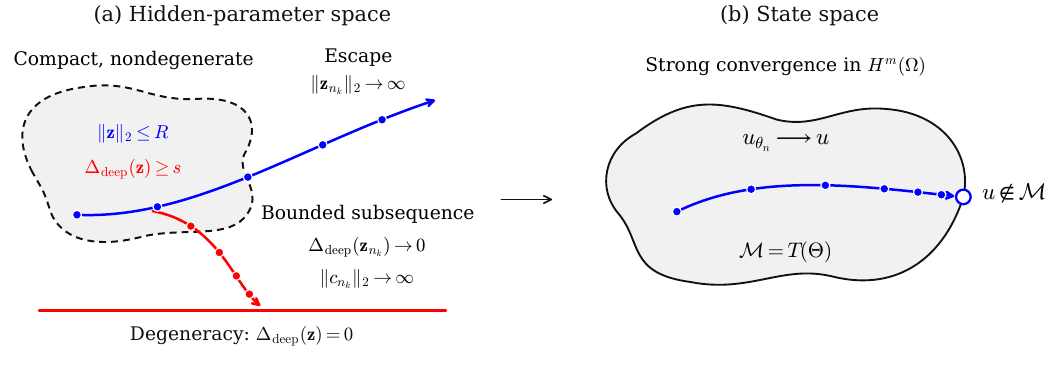}
\caption{Network outputs can converge while parameters escape or neurons
degenerate. The limiting state lies outside the model. Schematic.}
\label{fig:forced_singularity}
\end{paperfigure}
\subsection{Derivative completion captures all weak limits}
\label{sec:feature_collision}
\label{subsec:mechanisms}

For the classes \(\mathcal M_P\) and \(\mathcal C_P\) defined in
Section~\ref{subsec:preliminaries}, derivative blocks account for every
missing limit.

\Needspace{10\baselineskip}
\begin{theorem}[Exact completion and weak sequential closedness]
\label{thm:translated_feature_geometry}
Let \(m\ge1\) and \(P\ge1\) be integers, and let
\(0\neq K\in H^{m+P-1}(\mathbb R)\) be real-valued. The strong closure
\(\mathcal C_P=\overline{\mathcal M_P}^{\,H^m(\mathbb R)}\) consists exactly
of the sums
\[
u=\sum_{q=1}^{Q}\sum_{r=0}^{\nu_q-1}w_{q,r}K^{(r)}(\cdot-b_q),
\qquad \sum_{q=1}^{Q}\nu_q\le P.
\]
Here \(Q\) is an integer with \(0\le Q\le P\), the centres \(b_q\) are distinct real numbers,
\(w_{q,r}\in\mathbb R\), and \(\nu_q\in\mathbb N\).
The empty sum \(Q=0\) denotes the zero function.

In fact, every weakly convergent sequence in \(\mathcal M_P\) has a limit of this
form, and every such sum is a strong \(H^m(\mathbb R)\) limit of sequences
in \(\mathcal M_P\). Moreover, \(\mathcal C_P\) is weakly sequentially
closed in \(H^m(\mathbb R)\).
\end{theorem}

The proof is in
Appendix~\ref{subsec:translated_feature_geometry_proof}.

The weak sequential closedness of \(\mathcal C_P\) ensures that weak limits
of minimizing sequences remain in the completed class, so coercivity and
weak lower semicontinuity yield attainment by the direct method.

\Needspace{8\baselineskip}
\begin{corollary}[Variational attainment on the completed class]
\label{cor:completed_variational_attainment}
\label{rem:completed_variational_attainment}
Under the kernel hypotheses of Theorem~\ref{thm:translated_feature_geometry},
set \(H=H^m(\mathbb R)\) and
\(\mathcal C_P:=\overline{\mathcal M_P}^{\,H}\).
Every coercive, weakly lower semicontinuous \(J:H\to\mathbb R\)
attains its minimum on \(\mathcal C_P\). If \(J\) is also continuous in the
strong topology, then
\[
\inf_{u\in\mathcal M_P}J(u)=\min_{u\in\mathcal C_P}J(u).
\]
\end{corollary}
If \(J\) is strongly continuous, the original infimum is attained in the
completed class, even when the best approximation error is positive.
This does not imply uniqueness or
convergence of an optimizer. Appendix~\ref{subsec:attainment_corollaries_proof}
gives the proof.

\begin{remark}[Sequences approaching the full-space minimum]
\label{rem:translated_arbitrary_sequences}
Under the kernel hypotheses of Theorem~\ref{thm:translated_feature_geometry}
and Assumption~\ref{asu1}, let \(u^\ast\) be the full-space minimizer of \(J\).
Every sequence \(u_n\in\mathcal M_P\) with \(J(u_n)\to J(u^\ast)\) converges
strongly to \(u^\ast\), which then has the stated derivative representation.
Such a sequence exists exactly when \(u^\ast\in\mathcal C_P\), equivalently
when \(\inf_{v\in\mathcal M_P}J(v)=J(u^\ast)\). This zero-gap condition is
unnecessary for the constrained attainment in
Corollary~\ref{cor:completed_variational_attainment}.
The proof is in Appendix~\ref{subsec:translated_feature_geometry_proof}.
\end{remark}

Figure~\ref{fig:completion_patterns} illustrates the collision mechanism
with three Gaussian translates. The coefficients are fitted by least squares
at prescribed centres.
\begin{paperfigure}[!t]
\centering
\includegraphics[width=\linewidth]{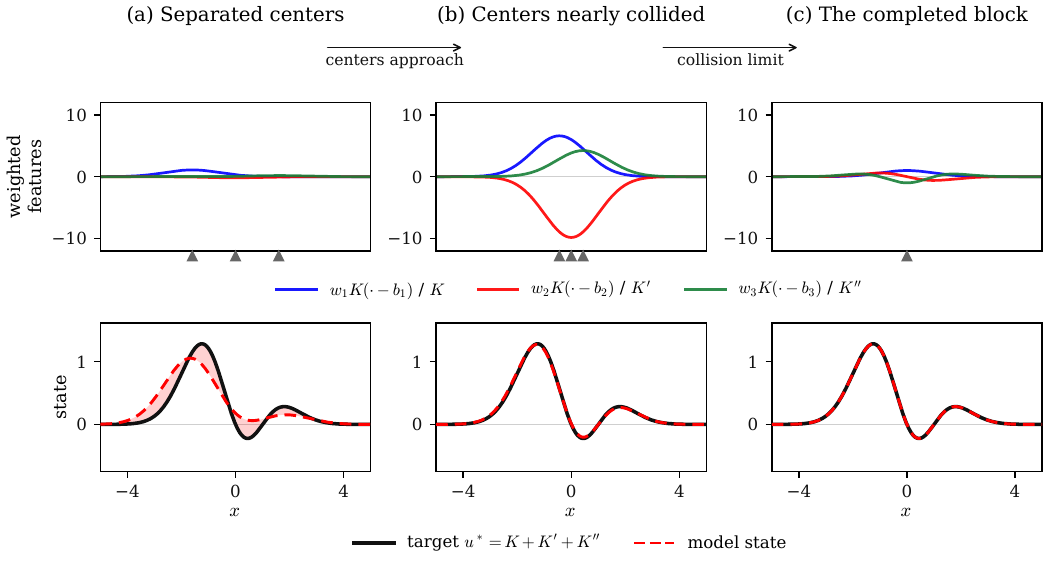}
\caption{As three Gaussian centres merge, weights grow and the fit improves.
A derivative block represents the target exactly. Top: individual
contributions; bottom: their sum and the target.}
\label{fig:completion_patterns}
\end{paperfigure}
\begin{remark}[The whole-line norm is essential]
On \(\Omega=(-1,1)\), the Cauchy kernel \(K(x)=(1+x^2)^{-1}\) satisfies
\[
b^2K(x-b)=\left(1-\frac{2x}{b}+\frac{x^2+1}{b^2}\right)^{-1}
\longrightarrow1\quad\text{in }H^m(\Omega)\quad(b\to+\infty).
\]
Convergence holds in \(C^m([-1,1])\). No finite sum of finite-centre
translates and their derivatives equals \(1\) on \(\Omega\): analyticity
would extend the equality to \(\mathbb R\), where every such sum decays.
Thus derivative completion alone need not restore bounded-domain
attainment.
\end{remark}

\subsection{Robust nonattainment and sharp coefficient growth}
\label{subsec:robust_nonattainment}

Fix integers \(m\ge1\), \(P\ge2\) and a real-valued kernel
\(0\ne K\in H^{m+\max\{P-1,5\}}(\mathbb R)\), and set
\(H=H^m(\mathbb R)\). Consider the simple-collision state
\[
u_0=a_0K(\cdot-c_0)+b_0K'(\cdot-c_0)
 +\sum_{j=1}^{P-2}d_{j,0}K(\cdot-c_{j,0}),
\]
with all centres distinct and \(b_0,d_{j,0}\ne0\).

The next theorem shows that nonattainment persists under target perturbations
and quantifies the coefficient growth needed to approach the best loss.

\Needspace{8\baselineskip}
\begin{theorem}[Robust nonattainment and sharp coefficient growth]
\label{thm:robust_nonattainment}
Every neighbourhood of \(u_0\) contains a nonempty open set \(U\subset H\)
such that, for each \(g\in U\), the functional associated with an elliptic-type energy,
\[
J(v):=\frac12\|v-g\|_H^2,
\]
has a unique minimizer
\(u_\ast\notin\mathcal M_P\) on \(\mathcal C_P\), with
\(\inf_{v\in\mathcal M_P}J(v)=J(u_\ast)\). Moreover,
\begin{equation}\label{eq:robust_coefficient_budget}
\lim_{R\to\infty}R^2
\inf_{\substack{u=u_\theta,\ \theta=(w,z)\in\mathbb R^{2P}\\ \|w\|_2\le R}}
\bigl[J(u)-J(u_\ast)\bigr]=C\in(0,\infty),
\end{equation}
where \(u_\theta\) is given by~\eqref{rbfnn}, with unrestricted centres.
\end{theorem}
Nonattainment persists for an open set of nearby targets: completion makes
their best approximation available. The original model can approach the
same loss only with growing coefficients.

\begin{remark}[The sharp constant]
\label{rem:robust_constant}
For a fixed \(g\in U\), let \(a,b,c\) be the coefficients and centre of
the collision block \(aK(\cdot-c)+bK'(\cdot-c)\) in \(u_\ast\). Then
\[
C=\frac{\ell b^2}{2},\qquad
\ell:=\left\langle u_\ast-g,
\frac a2K''(\cdot-c)+\frac b6K'''(\cdot-c)
\right\rangle_H>0.
\]
\end{remark}

Appendix~\ref{app:robust_nonattainment} gives the local description and
proofs; Section~\ref{subsec:manufactured_1d} tests the sharp law and the
faster decay of the excess error when the completed model represents the
target exactly.

\subsection{Transfer to derivative blocks}
\label{subsec:completion_algorithm}

We now describe a completion procedure that replaces nearly colliding
features with derivative blocks, starting with a pair of translated kernels.

\begin{example}[A two-slot replacement]
\label{ex:two_slot_transfer}
For the prescribed pairs used below, let \(K:\mathbb R^d\to\mathbb R\)
be a smooth kernel, \(c\in\mathbb R^d\) a centre, and
\(e\in\mathbb R^d\) a unit direction. Define the spatial directional derivative
\(D_eK(y):=e\cdot\nabla K(y)\) and write
\[
u_h(x)=w_+K(x-c-he)+w_-K(x-c+he),\qquad h>0,\quad |e|=1.
\]
Here \(h\) is the center half-gap, so the two centers are \(2h\) apart.
Taylor expansion suggests the two-slot replacement
\begin{equation}\label{eq:experimental_pair_transfer}
 u_h\ \longmapsto\ aK(\cdot-c)-bD_eK(\cdot-c),
 \qquad a=w_++w_-,\quad b=h(w_+-w_-).
\end{equation}
\end{example}

\noindent\textbf{Multiplicity budget.}
In the completed representation, a derivative block through order \(r\),
\[
\sum_{j=0}^{r}a_jK^{(j)}(\cdot-b),
\]
uses \(r+1\) slots, including slots whose coefficients vanish. Thus a single
\(K'\) uses two slots and a single \(K''\) uses three. The total cost of
all blocks must not exceed the original budget \(P\).

Completion replaces a collapsing cluster by derivative features within the
same multiplicity budget. For \(u\in\mathcal C_P\), the replacement in
Algorithm~\ref{alg:basic_completion} is constructed as follows:
\begin{enumerate}[label=\arabic*.,leftmargin=2.5em,labelsep=0.6em,itemsep=3pt]
\item Select a cluster \(C\) of whole feature blocks. Let \(u_C\) be its
contribution to the current state and \(\nu\) its total slot count.
\item Choose a centre \(\beta\) between the cluster's extreme centres.
\item Expand each translation about \(\beta\), retaining derivative orders
below \(\nu\):
\begin{equation}\label{eq:feature_taylor_replacement}
 wK^{(r)}(\cdot-c)\ \longmapsto\
 w\sum_{s=r}^{\nu-1}\frac{(\beta-c)^{s-r}}{(s-r)!}
 K^{(s)}(\cdot-\beta).
\end{equation}
\item Add the expanded terms and collect equal derivative orders to form
a single block \(T_C\) with at most \(\nu\) slots.
\end{enumerate}
This is the moment transfer in Proposition~\ref{prop:mixed_completion}.

Algorithm~\ref{alg:basic_completion} accepts this replacement when the
budget and error tests pass. It assumes an estimator
\(\operatorname{err}(u_C,T_C)\ge\|u_C-T_C\|_{H^m}\) for the
actually computed and stored trial, including conversion, evaluation and
rounding errors. Such an estimator is not constructed here.
Appendix~\ref{app:stable_transfer_rule} discusses accurate evaluation of
the same replacement and the additional bounds needed for validation.

\begin{algorithm}[H]
\caption{Completion of a feature cluster}
\label{alg:basic_completion}
\algrule
\begin{algorithmic}[1]
\Require A state \(u\in\mathcal C_P\); a cluster \(C\) using \(\nu\le P\)
slots; a centre \(\beta\) between its extreme centres; a tolerance
\(\tau\ge0\); and a transfer-error estimator \(\operatorname{err}\)
\Ensure Updated state and error estimate
\State Form and store \(T_C\) by~\eqref{eq:feature_taylor_replacement},
collecting equal derivative orders
\State \(u^{\rm comp}\gets u-u_C+T_C\)
\State \(e\gets\operatorname{err}(u_C,T_C)\)
\If{\(\#\text{slots}(u^{\rm comp})\le P\) and \(e\le\tau\)}
  \State \Return \((u^{\rm comp},e)\)
\Else
  \State \Return \((u,\infty)\)
\EndIf
\end{algorithmic}
\end{algorithm}

The experiments accept the replacement in~\eqref{eq:experimental_pair_transfer}
when the pair is close and the
measured absolute \(H^1\) change is small, then refit the coefficients and
resume optimization. The general transfer analysis and further computation details
are given in Appendix~\ref{app:mixed_completion_proof}.

\FloatBarrier

\section{Numerical tests of parameter growth and completion}\label{numerics}

Four experiments track neuron collisions, check the coefficient-growth law,
and compare standard optimization with completion on two manufactured PDEs.

\subsection{Training approaches singular limits}
\label{subsec:exp_escape}

On \((-10,10)\), two-neuron Deep Ritz fits minimize a trapezoidal
approximation of the energy
\begin{equation}\label{eq:numerics_tanh_objective}
 J(v)=20\int_{-10}^{10}
 \left[\frac12\bigl(|v'|^2+|v|^2\bigr)-fv\right]dx,
 \qquad v\in H^1(-10,10).
\end{equation}
The targets are \(u^*(x)=\operatorname{sech}^2x\), with
\(f=-{u^*}''+u^*\), and \(v^*(x)=x\operatorname{sech}^2x\), with
\(f=-{v^*}''+v^*\). The factor \(20\), the interval length, only rescales
the energy. Without a boundary load, the target need not minimize the
finite-domain Neumann problem, although its endpoint flux is exponentially
small. The first fit trains weights and biases; the second also trains
slopes.

\begin{paperfigure}[H]
\centering
\includegraphics[width=.78\linewidth]{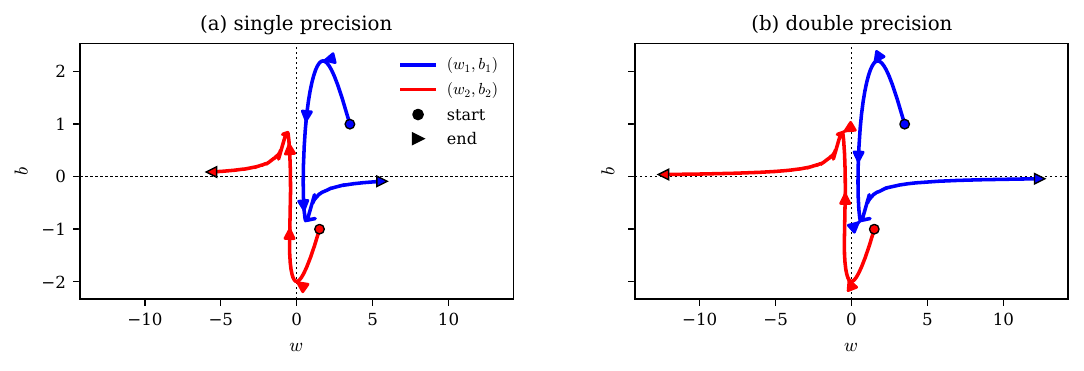}
\caption{Two tanh neurons merge during training as their weights grow with
opposite signs. Circles mark starts; triangles mark endpoints.
\textbf{(a)} Single precision; \textbf{(b)} double precision.}
\label{fig:tanh_param_traj}
\end{paperfigure}

In both, neurons coalesce and
weights grow in magnitude with opposite signs, while the network output
approaches the target
(Figures~\ref{fig:tanh_param_traj}--\ref{fig:tanh_neuron_evolution}).
These trajectories illustrate the missing-limit mechanism.

\begingroup
\setlength{\intextsep}{4pt}
\begin{paperfigure}[H]
\centering
\includegraphics[width=.92\linewidth]{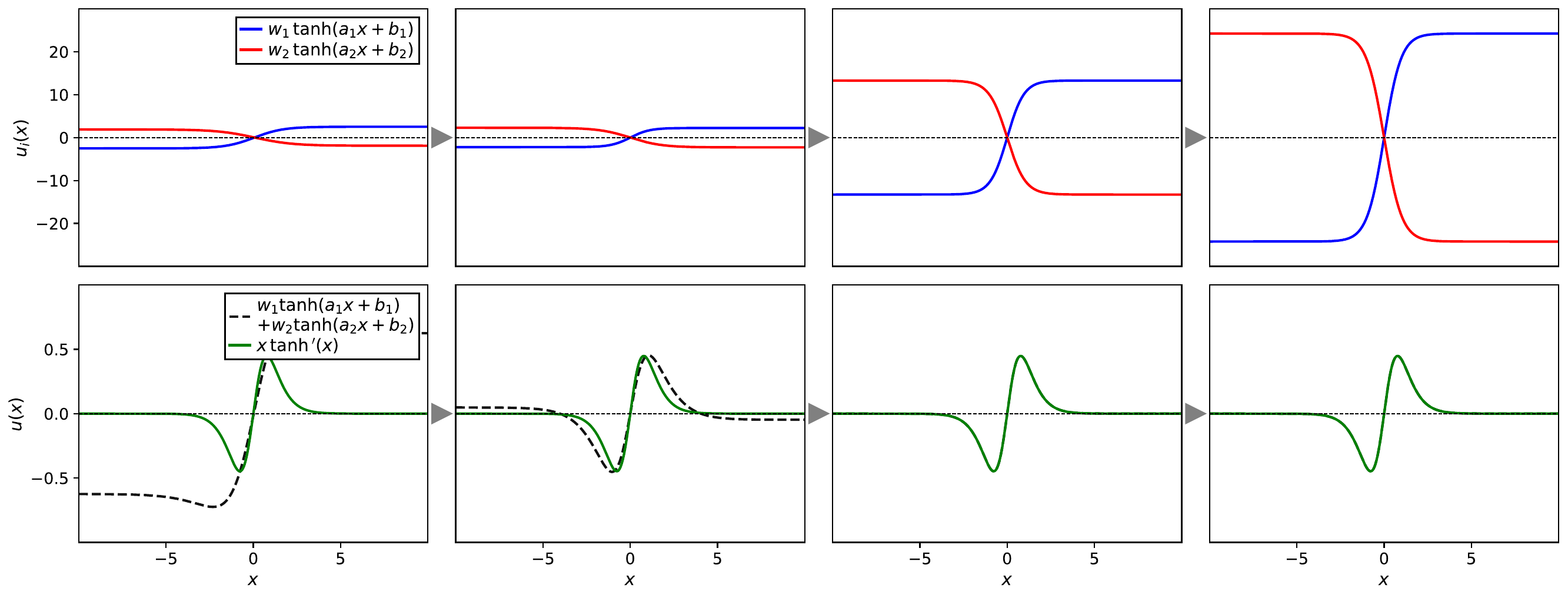}
\caption{Large neuron contributions nearly cancel, bringing their sum closer
to the target. Top: contributions; bottom: sum (dashed) and target (green).
The last two columns show independent refinements.}
\label{fig:tanh_neuron_evolution}
\end{paperfigure}
\endgroup

\subsection{Coefficient caps reveal the predicted growth law}
\label{subsec:manufactured_1d}

This experiment tests the coefficient-growth law of
Theorem~\ref{thm:robust_nonattainment}. For \(K(x)=e^{-x^2/2}\), the
manufactured modified Helmholtz solutions are
\[
g_\varepsilon=K'-\varepsilon q,\qquad
q=\tfrac16(K'''+\tfrac{21}{10}K'),\qquad
\varepsilon\in\{0,10^{-3},10^{-2}\}.
\]
For each fixed \(\varepsilon\), we minimize the energy \(J\) of
Theorem~\ref{thm:robust_nonattainment} with \(H=H^1(\mathbb R)\),
\begin{equation}\label{eq:manufactured_1d_energy}
 J(v)=\frac12\|v-g_\varepsilon\|_{H^1(\mathbb R)}^2
 =\int_{\mathbb R}\Bigl[\frac12\bigl(|v'|^2+|v|^2\bigr)
 -f_\varepsilon v\Bigr]dx+\frac12\|g_\varepsilon\|_{H^1(\mathbb R)}^2,
\end{equation}
where \(f_\varepsilon=g_\varepsilon-g_\varepsilon''\). Up to an additive
constant, this is the Ritz energy of \(-u''+u=f_\varepsilon\).

We search over sums of two shifted Gaussians, imposing the coefficient bound
\(\|w\|_2\le R\). We use 29 values of \(R\) between \(2\) and \(10^5\)
and refine the candidates using 80-digit arithmetic.
Figure~\ref{fig:manufactured_1d} compares the resulting approximations with
the completed reference \(u_0=K'\) and a prescribed difference quotient.

\begin{paperfigure}[H]
\centering
\includegraphics[width=\linewidth]{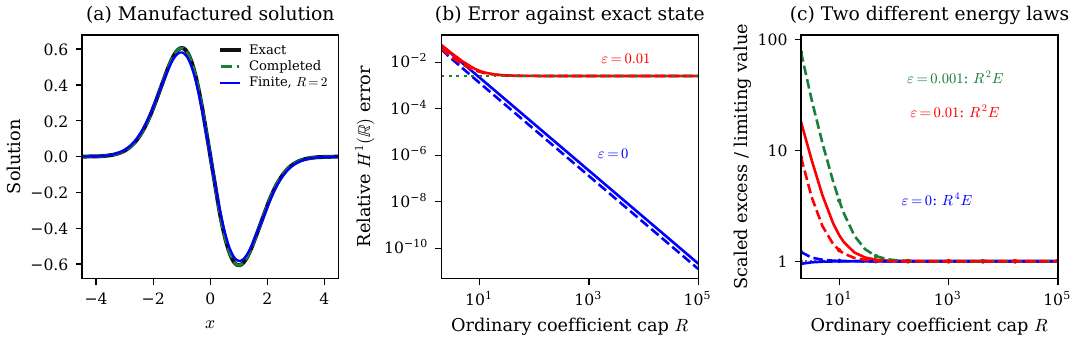}
\caption{Larger coefficient budgets improve the Gaussian approximation.
\textbf{(a)} States; \textbf{(b)} errors; \textbf{(c)} rescaled excess energies
approach their predicted limits. Dotted line: completed error for \(\varepsilon=0.01\).}
\label{fig:manufactured_1d}
\end{paperfigure}

Let \(u_R\) denote the best ordinary pair found under the cap. The computed
excess energies \(E_{\rm num}(R):=J(u_R)-J(u_0)\) follow the predicted scalings:
\[
E_{\rm num}(R)\sim
\begin{cases}
\dfrac{13\sqrt\pi}{1280}R^{-4}, & \varepsilon=0,\\[10pt]
\dfrac{\varepsilon\|q\|_{H^1}^2}{2}R^{-2}, & \varepsilon>0.
\end{cases}
\]
For \(\varepsilon=0\), completion represents the exact target with
coefficient norm one. At \(R=10^5\), the best ordinary pair found has
relative \(H^1\) error \(1.275\times10^{-11}\). For positive
\(\varepsilon\), the completed reference has nonzero error, while the
computed excess energies follow the predicted \(R^{-2}\) law.

\Needspace{5\baselineskip}
\subsection{Completion for a Poisson--Neumann problem}
\label{subsec:poisson_2d_experiment}

A Poisson--Neumann problem on \(\Omega=(-1,1)^2\) tests a model with
four Gaussian pairs and two singletons (ten slots). The problem
\(-\Delta u=f\), \(\partial_\nu u=g\) is solved by minimizing
\begin{equation}\label{eq:poisson_energy}
 J(u)=\int_\Omega\Bigl[\tfrac12|\nabla u|^2-fu\Bigr]dx
 -\int_{\partial\Omega}gu\,dS
 +\frac{\sigma}{2}\Bigl(\int_\Omega u\,dx-M\Bigr)^{2},
\end{equation}
with \(\sigma=1\). The data \(f=-\Delta u^\ast\), \(g=\partial_\nu u^\ast\)
and \(M=\int_\Omega u^\ast\) are manufactured from a target \(u^\ast\) in
the completion.

The mean term fixes the additive constant left undetermined
by the Neumann data. By the Poincar\'e--Wirtinger inequality,
\(\|\nabla u\|^2+\sigma(\int_\Omega u)^2\) controls \(\|u\|_{H^1}^2\).
Thus \(J\) satisfies Assumption~\ref{asu1} and \(u^\ast\) uniquely
minimizes it.

Figure~\ref{fig:poisson_2d_fields} shows start~16.
Pair separations are enlarged for visibility; connectors
mark the true centres. Error maps share a logarithmic scale, with values
below \(10^{-12}\) clipped for display.

\begin{paperfigure}[!ht]
\centering
\includegraphics[width=\linewidth]{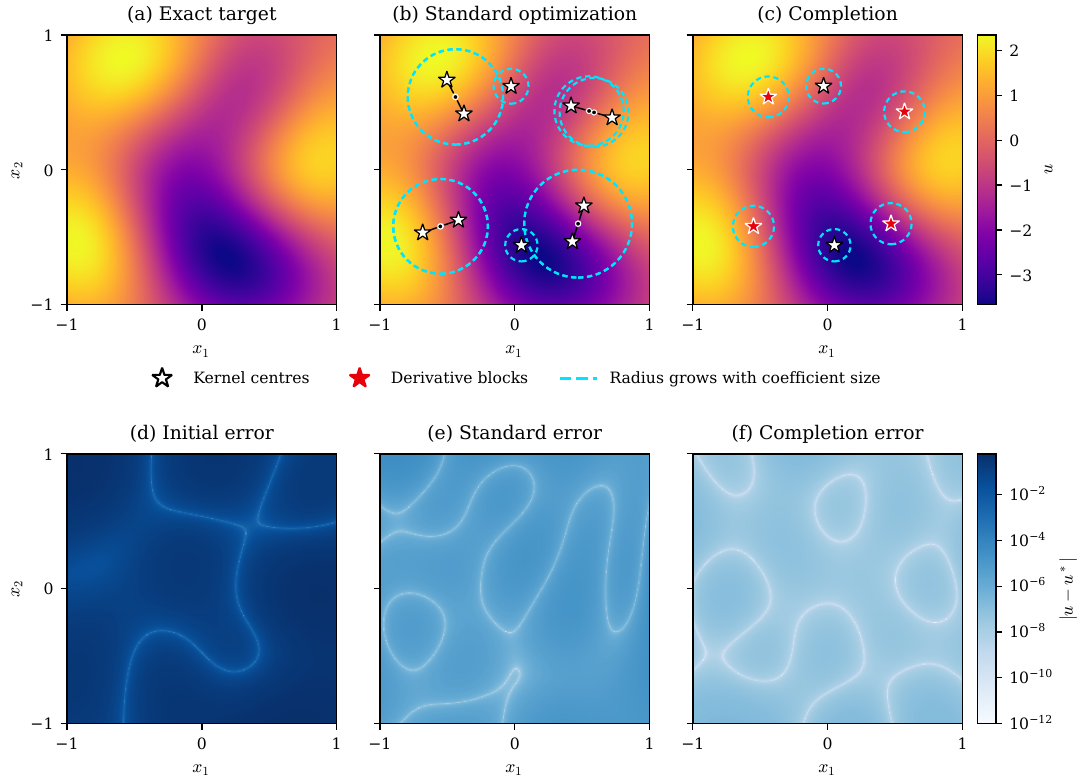}
\caption{Completion sharply reduces the Poisson error in a representative run.
Top: target and final approximations; bottom: error maps.
Larger circles indicate larger coefficients.}
\label{fig:poisson_2d_fields}
\end{paperfigure}

Both two-dimensional experiments compare \emph{standard optimization} and
\emph{completion} from ten shared local starts. For the Poisson problem,
these starts lie near the target centres
(Appendix~\ref{app:numerics_poisson}). At positive separation,
both fit the ordinary weights \(w_+,w_-\) against directly evaluated
translates in double precision. Coefficients are refitted at each geometry
(variable projection); L-BFGS-B optimizes centres, directions and \(\log h\)
with the same settings and an allowance of 8,000 coefficient fits.

Standard optimization retains every finite pair. Completion applies
\eqref{eq:experimental_pair_transfer} to eligible pairs, refits, and restarts
L-BFGS-B; completed blocks use kernel and directional-derivative features
directly. Pairs are fixed in advance, and all starts, including abnormal
optimizer returns, are retained at their last accepted state
(Appendix~\ref{app:numerics_raw_main}).
Reported work counts every coefficient fit through termination, including
line-search and transfer evaluations.

The PDE comparisons use ordinary weights and direct evaluation of finite
kernel translates. Supplementary evaluation tests and optimizer controls
are collected in Appendix~\ref{sec:numerics_details}, with reproduction
instructions in the \href{https://github.com/danielfdzm/singular-representations-pde}{code repository}.

Both experiments use tensor-product Gauss--Legendre quadrature. The order gives
the number of points per coordinate and on each boundary edge; its square
gives the number of interior points.
For the Poisson problem, training uses quadrature of order 32; reported errors
use independent quadrature of order 96.

A coefficient norm always means the actual weights of
the current representation: ordinary kernel weights and
kernel/derivative coefficients for completed pairs.

\FloatBarrier

\begin{table}[H]\centering\footnotesize
\setlength{\tabcolsep}{4pt}
\caption{Final accuracy, coefficient sizes, and work for the Poisson--Neumann
problem over ten shared starts.}
\label{tab:poisson_raw_results}
\begin{tabular}{@{}lcc@{}}\toprule
 & Standard optimization & Completion\\\midrule
Relative \(H^1\) error & \(9.8\times10^{-6}\text{--}6.2\times10^{-5}\) & \(1.7\times10^{-8}\text{--}2.6\times10^{-7}\)\\
\quad median & \(1.8\times10^{-5}\) & \(8.1\times10^{-8}\)\\
Coefficient norm & \(2.3\times10^{2}\text{--}5.0\times10^{3}\) & \(3.53694\text{--}3.53698\)\\
Smallest center half-gap & \(2.1\times10^{-4}\text{--}3.1\times10^{-3}\) & \(0\)\\
Solves (median) & 758--1785 (1204.5) & 377--525 (431)\\
Completed pairs & 0 & 4\\
Normal optimizer returns & 5/10 & 10/10\\
\bottomrule\end{tabular}
\end{table}

Completion transfers all four pairs on every start and reduces the median
relative \(H^1\) error from \(1.8\times10^{-5}\) to
\(8.1\times10^{-8}\) (Table~\ref{tab:poisson_raw_results}). Its final
coefficient norms are close to the completed target's \(3.5369\);
ordinary weights reach a norm of \(5.0\times10^3\). Five standard runs
have abnormal line-search returns; all completion runs terminate normally.

Figure~\ref{fig:poisson_2d_trajectories} includes every start.

\begingroup
\setlength{\intextsep}{4pt}
\begin{paperfigure}[H]
\centering
\includegraphics[width=.94\linewidth]{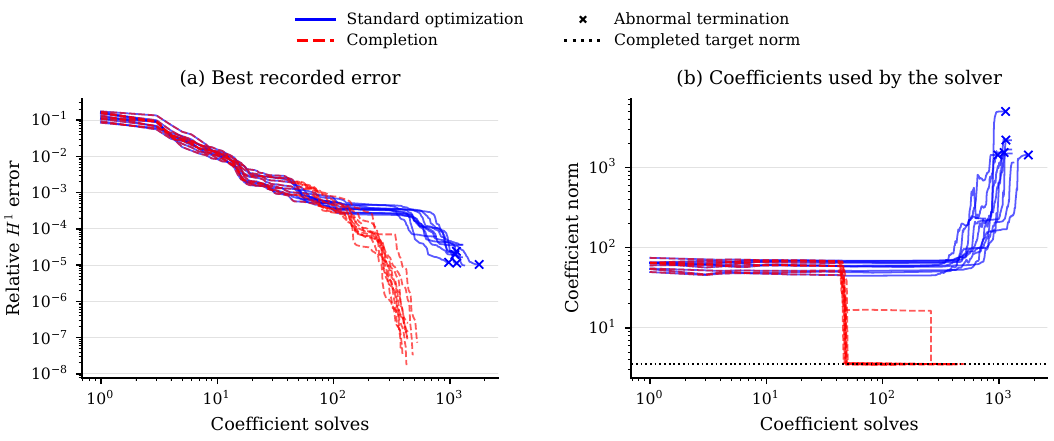}
\caption{Completion improves Poisson accuracy and controls coefficient growth
across ten starts. Left: best relative error so far; right: coefficient norm.
Crosses mark abnormal termination.}
\label{fig:poisson_2d_trajectories}
\end{paperfigure}
\endgroup

For the Poisson problem, completion reaches relative error \(10^{-5}\) from all ten starts
(median 291.5 solves); standard optimization reaches it from one start,
after 1,283 solves (Table~\ref{tab:raw_common_accuracy_work}). Median total work is
431 solves for completion and 1,204.5 for standard optimization. The
higher-order training comparison gives similar median errors
(Table~\ref{tab:raw_training_quadrature}).

\subsection{Completion with Mat\'ern features}
\label{subsec:matern_2d_experiment}

We test a nonlinear fourth-order Swift--Hohenberg problem on
\(\Omega=(-1,1)^2\) with two Mat\'ern pairs and one singleton (five slots).
We minimize the energy
\begin{equation}\label{eq:matern_energy}
 J(u)=\int_\Omega\Bigl[\tfrac12\bigl((1+\Delta)u\bigr)^2+\tfrac{\rho}{2}u^2
 +\tfrac14u^4-fu\Bigr]\,dx
 -\int_{\partial\Omega}\Bigl[m\,\partial_\nu u+qu\Bigr]\,dS,
\end{equation}
with \(\rho=1\) and natural boundary conditions.

The Mat\'ern energy is strictly convex on \(H^2(\Omega)\), since
\[
 D^2J(u)[v,v]=\|(1+\Delta)v\|_{L^2}^2+\|v\|_{L^2}^2
 +3\int_\Omega u^2v^2\,dx\geq\|v\|_{L^2}^2.
\]
The loads
\(f=(1+\Delta)^2u^\ast+u^\ast+(u^\ast)^3\), \(m=(1+\Delta)u^\ast\) and
\(q=-\partial_\nu((1+\Delta)u^\ast)\) are manufactured from a target
\(u^\ast\) in the completed model. They make \(u^\ast\) stationary and
hence the unique minimizer.

The model uses translates of the \(C^8\) Mat\'ern kernel with smoothness
\(\nu_{\mathrm M}=9/2\) and length scale \(\ell=3/\sqrt7\),
\begin{equation}\label{eq:matern_kernel}
 \Phi(y)=e^{-s}\Bigl(1+s+\tfrac37s^2+\tfrac{2}{21}s^3+\tfrac{1}{105}s^4\Bigr),
 \qquad s=\sqrt7\,|y|,\quad y\in\mathbb R^2,
\end{equation}
whose regularity supports the fourth-order manufactured load. Its Fourier
transform decays like \((1+|\xi|^2)^{-(\nu_{\mathrm M}+1)}\) and never vanishes, so
Assumption~\ref{asu:kernel} holds \cite{Matern1960,Stein1999}.

The Mat\'ern quadratic-energy error is
\(e_Q=\|u-u^\ast\|_Q/\|u^\ast\|_Q\), where
\(\|v\|_Q^2=\|(1+\Delta)v\|_{L^2}^2+\|v\|_{L^2}^2\).

We use the shared optimization procedure described in Section~\ref{subsec:poisson_2d_experiment}.
Training uses quadrature of order 64; reported errors use independent
quadrature of order 96.

Completion reduces the median relative \(H^1\) error from
\(1.7\times10^{-5}\) to \(2.3\times10^{-7}\)
(Table~\ref{tab:matern_raw_results}).

\begin{table}[!ht]\centering\footnotesize
\setlength{\tabcolsep}{4pt}
\caption{Final accuracy, coefficient sizes, and work for the Mat\'ern problem
over ten shared starts.}
\label{tab:matern_raw_results}
\begin{tabular}{@{}lcc@{}}\toprule
 & Standard optimization & Completion\\\midrule
Relative \(H^1\) error & \(2.1\times10^{-6}\text{--}3.5\times10^{-5}\) & \(3.4\times10^{-8}\text{--}1.8\times10^{-5}\)\\
\quad median & \(1.7\times10^{-5}\) & \(2.3\times10^{-7}\)\\
Relative \(e_Q\) & \(4.8\times10^{-6}\text{--}7.4\times10^{-5}\) & \(2.5\times10^{-8}\text{--}3.4\times10^{-5}\)\\
Coefficient norm & \(7.0\times10^{1}\text{--}1.0\times10^{3}\) & \(2.5\times10^{0}\text{--}1.0\times10^{2}\)\\
Smallest center half-gap & \(7.5\times10^{-4}\text{--}1.3\times10^{-2}\) & \(0\)\\
Fits (median) & 149--681 (225.5) & 108--713 (255.5)\\
Completed pairs & 0 & 1--2\\
Normal optimizer returns & 9/10 & 10/10\\
\bottomrule\end{tabular}
\end{table}

Nine starts transfer both pairs;
start~1 transfers only one and ends at \(1.8\times10^{-5}\), with
coefficient norm about \(100\). The remaining completion runs have
coefficient norms close to the target's \(2.4637\), whereas ordinary
weights reach a norm of \(1.0\times10^3\)
(Figure~\ref{fig:matern_2d_trajectories}). Standard optimization has one
abnormal line-search return; completion has ten normal returns.

\FloatBarrier

Higher-order training is checked in
Table~\ref{tab:raw_training_quadrature}; the failure to complete start~1
persists and its returned error increases.

\begin{paperfigure}[H]
\centering
\includegraphics[width=\linewidth]{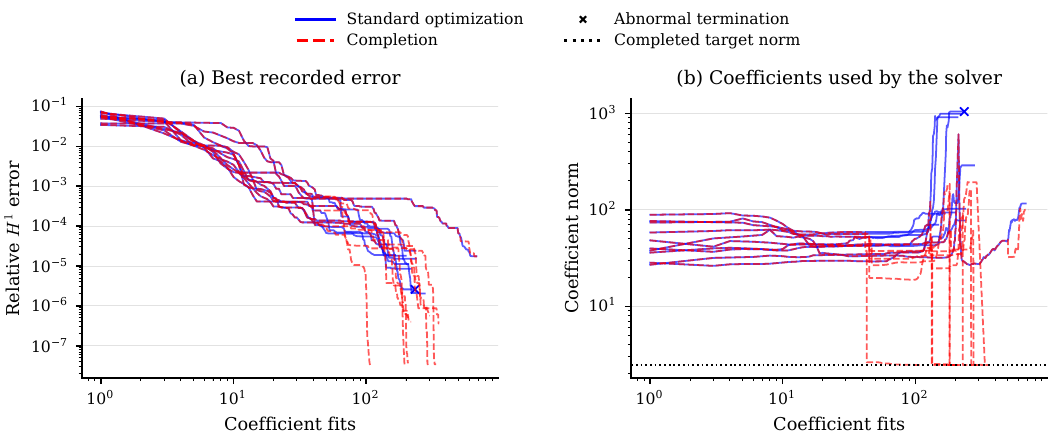}
\caption{Completion improves Mat\'ern accuracy and reduces coefficient growth.
Left: best relative error so far; right: coefficient norm.
Crosses mark abnormal termination.}
\label{fig:matern_2d_trajectories}
\end{paperfigure}

Nine completion starts reach \(10^{-5}\), against three for standard
optimization (Table~\ref{tab:raw_common_accuracy_work}). Median total
work is 255.5 coefficient fits for completion and 225.5 for standard
optimization. Completion improves accuracy here at a higher median cost.

\begin{table}[!ht]\centering\footnotesize
\setlength{\tabcolsep}{4pt}
\caption{Work needed to reach each relative \(H^1\) error threshold.
Counts are out of ten starts; medians include only starts that reach the threshold.}
\label{tab:raw_common_accuracy_work}
\begin{tabular}{@{}llrrrr@{}}\toprule
 & & \multicolumn{2}{c}{Standard optimization} & \multicolumn{2}{c}{Completion}\\
\cmidrule(lr){3-4}\cmidrule(l){5-6}
Problem & \(\tau\) & Reached & Median & Reached & Median\\\midrule
Poisson & \(1.0\times10^{-3}\) & 10/10 & 51.5 & 10/10 & 69\\
 & \(1.0\times10^{-5}\) & 1/10 & 1283 & 10/10 & 291.5\\
 & \(1.0\times10^{-7}\) & 0/10 & --- & 7/10 & 427\\
\midrule
Mat\'ern & \(1.0\times10^{-3}\) & 10/10 & 20 & 10/10 & 20\\
 & \(1.0\times10^{-5}\) & 3/10 & 169 & 9/10 & 180\\
 & \(1.0\times10^{-7}\) & 0/10 & --- & 5/10 & 201\\
\bottomrule\end{tabular}
\end{table}

\section{Discussion and perspectives}\label{conopen}

An accurate solution can have an unstable neural representation. The
coercivity gap explains how a decreasing loss can coexist with growing
weights and merging neurons: the model may be approaching a state that it
cannot represent. Completion makes these missing limits available as
model features.

For a fixed number of shifted kernels on the real line, derivative
completion captures every missing limit and makes the best approximation
attainable under the stated assumptions. The analysis also shows that
nonattainment in the original class
persists under small target perturbations and quantifies the weight growth
needed to approach the optimal loss. Finding the missing optimal states
remains a separate task for the optimizer.

The experiments track how ordinary optimization approaches singular
representations. Completion changes the model available to the optimizer,
but the loss remains nonconvex in the model parameters, so training is not
guaranteed to find the best state. The PDE comparisons use known solutions,
prescribed pairs and initializations near target configurations, so broader
performance remains to be established.

Future work should focus on three directions:
\begin{itemize}[leftmargin=2.5em,labelsep=0.6em]
\item \textbf{Automatic completion.}
Detect useful collisions during training, handle larger groups of neurons,
and decide when more features are needed.
\item \textbf{Reliable and efficient algorithms.}
Develop affordable checks for changes of representation, establish
convergence guarantees, and test accuracy and cost on broader PDE benchmarks.
\item \textbf{Richer models.}
Identify missing limits in higher dimensions and deeper networks, then
design representations that include them.
\end{itemize}

\section*{Acknowledgements}
The author acknowledges funding from Schaeffler Technologies AG \& Co. KG.
The author is deeply grateful to Enrique Zuazua for his constant support and his encouragement
to bring this work to publication, and thanks Christopher Straub (Fraunhofer IISB,
Erlangen) for a careful critical reading of an earlier version of the manuscript.


\appendix
\fontsize{10}{12}\selectfont

\section{Proofs of the main results}\label{sec:proofs}

\subsection{Deep neural tanh-networks}
\label{app:deep_tanh_proof}

\begin{proof}[Proof of Theorem~\ref{thm:tanh_forced_collision}]
Let \(\alpha>0\) be a modulus of strong convexity for \(J\), as in
Assumption~\ref{asu1}. The midpoint inequality and the minimality of \(u^\ast\) give
\[
\frac{\alpha}{8}\|u_{\theta_n}-u^\ast\|_{H^m(\Omega)}^2
\le\frac12\bigl(J(u_{\theta_n})-J(u^\ast)\bigr)\longrightarrow0.
\]
Thus \(u_{\theta_n}\to u^\ast\) strongly in \(H^m(\Omega)\). It remains to
prove the assertion of Remark~\ref{rem:deep_arbitrary_sequences}.
Accordingly, assume that \(u_{\theta_n}\to u\) strongly and that \(u\) is
not representable by the fixed architecture. Define
\(T_{\mathbf z}c:=c_0+\sum_{j=1}^{p_L}c_jh^L_{\mathbf z,j}\big|_\Omega\),
with \(c\) ranging over \(V=\mathbb R^{p_L+1}\) when an output bias is
allowed and over \(V=\{0\}\times\mathbb R^{p_L}\) otherwise.

\begin{step1}[Injectivity of the output map]
The null-net theorem for deep neural tanh-networks states that a nontrivial regular
network cannot have identically zero output, even with an output constant
\cite[Theorems~3--4]{VlacicBolcskei2021}. Regularity requires every hidden
neuron to have a path to an active output and prohibits identical or
sign-opposite incoming parameters for neurons with the same parents
\cite[Lemma~1]{VlacicBolcskei2021}.

Suppose \(\Delta_{\mathrm{deep}}(\mathbf z)>0\) and \(T_{\mathbf z}c=0\)
in \(H^m(\Omega)\). If some \(c_j\), \(j\ge1\), is nonzero, retain those
last-layer neurons and all their ancestors along nonzero-weight edges.
Delete the other hidden neurons and all zero-weight edges. This preserves
the output, and every retained hidden neuron has a path to an active output.
Since \(\Delta_{\mathrm{deep}}(\mathbf z)>0\), every incoming row is nonzero.
Thus each retained hidden neuron has at least one parent, and all its
nonzero incoming edges remain. Induction over the layers shows that its
graph level is still its original layer.

Two retained neurons with the same parent set therefore lie in the same
layer. If their incoming parameters agreed up to sign, their original
augmented rows, padded with zeros outside that parent set, would also
agree up to sign. This contradicts \(\Delta_{\mathrm{deep}}(\mathbf z)>0\).
The pruned network is consequently regular and nontrivial. Its output is
zero almost everywhere on \(\Omega\), hence everywhere on \(\Omega\) by
continuity and on \(\mathbb R^d\) by real analyticity. This contradicts the
null-net theorem. If instead all \(c_j=0\) for \(j\ge1\), then
\(T_{\mathbf z}c=c_0=0\). In either case the kernel of \(T_{\mathbf z}\)
is trivial.
\end{step1}

\begin{step2}[A uniform bound on the output coefficients]
The spatial derivatives through order \(m\) depend continuously on the
parameters, uniformly on a compact box containing \(\Omega\). Hence
\(\mathbf z\mapsto h^L_{\mathbf z,j}\) is continuous into
\(H^m(\Omega)\), and \((\mathbf z,c)\mapsto T_{\mathbf z}c\) is jointly
continuous. For \(R,s>0\), set
\[
K_{R,s}:=\{\mathbf z:\|\mathbf z\|_2\le R,\
\Delta_{\mathrm{deep}}(\mathbf z)\ge s\}.
\]
This set is compact. If it is nonempty, continuity and Step~1 give
\[
\sigma_{R,s}:=
\min_{\substack{\mathbf z\in K_{R,s}\\ c\in V,\ \|c\|_2=1}}
\|T_{\mathbf z}c\|_{H^m(\Omega)}>0.
\]
Indeed, the minimum is attained on a compact set, and a zero minimum would
contradict injectivity. Scaling \(c\) then gives
\begin{equation}\label{eq:deep_uniform_readout_bound}
 \|c\|_2\le C_{R,s}\|T_{\mathbf z}c\|_{H^m(\Omega)}
 \qquad(\mathbf z\in K_{R,s},\ c\in V).
\end{equation}
Here \(C_{R,s}=\sigma_{R,s}^{-1}\); the estimate is vacuous if
\(K_{R,s}\) is empty.
\end{step2}

\begin{step3}[Parameter divergence and degeneracy]
If \(\|\theta_n\|_2\) did not tend to infinity, a bounded subsequence
would have a further subsequence
converging to some \((c_*,\mathbf z_*)\) with
\(u=T_{\mathbf z_*}c_*\) by continuity. This contradicts the
nonrepresentability of \(u\) and proves (i).

Now take a subsequence with \(\|\mathbf z_{n_k}\|_2\le R\). Part (i)
forces \(\|c_{n_k}\|_2\to\infty\). If
\(\Delta_{\mathrm{deep}}(\mathbf z_{n_k})\) did not tend to zero, some
further subsequence would satisfy \(\Delta_{\mathrm{deep}}\ge s>0\) and
hence lie in \(K_{R,s}\). The states are bounded, so
\eqref{eq:deep_uniform_readout_bound} would bound the output coefficients
on that subsequence, a contradiction. This proves (ii).

Finally, if (iii) failed, there would be an \(\eta\in(0,1)\) and a
subsequence on which both
\[
\|\mathbf z_n\|_2\le\eta^{-1}-1,
\qquad \Delta_{\mathrm{deep}}(\mathbf z_n)\ge\eta.
\]
This contradicts (ii) and proves (iii).
\end{step3}
\end{proof}

\subsection{Cluster decomposition and closure}
\label{subsec:translated_feature_geometry_proof}
\label{subsec:attainment_corollaries_proof}

\subsubsection{Translations and derivatives}
\label{subsec:prelim_translation}

Throughout, \(m\ge1\) is an integer, and we write \(\tau_cv:=v(\cdot-c)\).
We use the unitary Fourier transform and the Sobolev norm
\[
\|v\|_{H^m(\mathbb R^d)}^2=\int p_m|\widehat v|^2\,d\xi,
\qquad p_m(\xi):=\sum_{|\alpha|\le m}\xi^{2\alpha}.
\]
For fixed \(v,w\in H^m\), the correlation
\(y\mapsto\langle\tau_yv,w\rangle_{H^m}\) belongs to
\(C_0(\mathbb R^d)\), since it is a finite sum of convolutions of
\(L^2\) functions.
Thus \(\tau_{y_n}v\rightharpoonup0\) whenever \(|y_n|\to\infty\).

For real-valued \(0\ne K\in H^{m+P-1}(\mathbb R)\), set \(F(t):=\tau_tK\).
Sobolev difference quotients show that \(F\in C^{P-1}(\mathbb R;H^m)\) and
\(F^{(r)}(t)=(-1)^r\tau_tK^{(r)}\). For \(r<P\), the Hermite--Genocchi
formula~\cite{deBoor1969} defines Newton coordinates for real nodes,
allowing repetitions:
\begin{equation}\label{eq:completed_newton_coordinates}
 E_r(t_1,\ldots,t_{r+1}):=r!\,[t_1,\ldots,t_{r+1}]F
 =\mathbb E\,F^{(r)}\Bigl(\sum_{j=1}^{r+1}S_jt_j\Bigr).
\end{equation}
Here \(S\) is uniformly distributed on the standard \(r\)-simplex and is
a point mass for \(r=0\). The coordinates remain continuous when nodes coincide:
\(E_r(b,\ldots,b)=(-1)^r\tau_bK^{(r)}\)
\cite{BatenkovYomdin2014,deBoor1969}.
The change from Hermite data \(F^{(i)}(b_q)\), \(i<\nu_q\), at distinct
centres to \(E_0,\ldots,E_{N-1}\), \(N=\sum_q\nu_q\), is triangular and
invertible.

These derivative families are linearly independent. Indeed, if
\(\lambda_1,\ldots,\lambda_Q\) are distinct, the derivatives involved lie in
\(H^m\), and \(\sum_q\sum_{r<\nu_q}\omega_{q,r}\tau_{\lambda_q}K^{(r)}=0\),
then Fourier transformation yields an exponential polynomial
\(\sum_qe^{-i\lambda_q\xi}\sum_{r<\nu_q}\omega_{q,r}(i\xi)^r\) vanishing
almost everywhere on the positive-measure set \(\{\widehat K\ne0\}\). The
identity theorem and the uniqueness of exponential polynomials imply
\(\omega=0\). Every finite Gram matrix of such a family is therefore
positive definite; no pointwise nonvanishing of \(\widehat K\) is needed.

\subsubsection{Cluster decomposition}

The following decomposition of bounded sequences in \(\mathcal M_P\)
will be used to prove Theorem~\ref{thm:translated_feature_geometry}.

\begin{lemma}[Cluster decomposition]
\label{lem:translated_cluster_decomposition}
Let \(0\ne K\in H^{m+P-1}(\mathbb R)\) be real-valued, and let
\((u_n)\subset\mathcal M_P\) be bounded in \(H^m(\mathbb R)\).
After passing to a subsequence, there is a finite decomposition
\[
u_n=\sum_A\tau_{c_{A,n}}g_{A,n},\qquad
g_{A,n}\longrightarrow g_A\quad\hbox{strongly in }H^m(\mathbb R),
\qquad \tau_cv=v(\cdot-c),
\]
where the reference centres of distinct clusters satisfy
\(|c_{A,n}-c_{A',n}|\to\infty\).
Each \(g_A\) is a finite sum of kernel derivatives at distinct relative
centres, and the total multiplicity, summed over all \(A\), is at most \(P\).
The reference centres have limits in \(\mathbb R\cup\{-\infty,+\infty\}\).
At most one reference centre has a finite limit \(c_A\); its contribution converges
strongly to \(\tau_{c_A}g_A\). Every other contribution converges weakly
to zero. The decomposition may be empty.
\end{lemma}

\begin{proof}[Proof of Lemma~\ref{lem:translated_cluster_decomposition}]
\begin{step1}[Cluster coordinates]
Combine terms with repeated centres and discard terms with zero coefficients.
After passing to a subsequence, we may assume that there are \(N\le P\)
distinct centres \(t_{j,n}\), and that all centres and pairwise differences
converge in \([-\infty,\infty]\).
The case \(N=0\) is immediate. Partition the indices into clusters
\(A\) according to whether \(t_{j,n}-t_{k,n}\) has a finite limit.
For each cluster, choose a reference centre
\(c_{A,n}=t_{j_A,n}\) and write \(h_{j,n}=t_{j,n}-c_{A,n}\to\lambda_j\).
The reference centres of distinct clusters separate to infinity. Within
\(A\), group indices with a common limiting offset \(\lambda_B\) into
collision subclusters \(B\in\mathfrak B(A)\).
For \(B=\{j_0,\ldots,j_{\nu_B-1}\}\) and \(0\le r<\nu_B\), set
\[
e_{B,r,n}=E_r(h_{j_0,n},\ldots,h_{j_r,n}),\qquad
E_{B,r,n}=\tau_{c_{A,n}}e_{B,r,n}.
\]
By Section~\ref{subsec:prelim_translation}, these functions span the
original translates, and \(e_{B,r,n}\to F^{(r)}(\lambda_B)\) in \(H^m\)
regardless of the relative collision rates.
\end{step1}

\begin{step2}[Bounded coefficients]
Linear independence within each cluster implies that its limiting Gram
matrix is positive definite. The Gram entries involving different clusters
tend to zero by correlation decay, convergence of the local bases, and
separation of the reference centres. Thus, for some \(\eta>0\) and all large \(n\),
\[
\Bigl\|\sum_{A,B,r}q_{B,r}E_{B,r,n}\Bigr\|_{H^m}
\ge\eta\Bigl(\sum_{A,B,r}|q_{B,r}|^2\Bigr)^{1/2}.
\]
Write \(u_n=\sum_{A,B,r}\gamma_{B,r,n}E_{B,r,n}\). The boundedness of
\((u_n)\) therefore implies that the coefficients are bounded. Passing to
a further subsequence, we may
assume \(\gamma_{B,r,n}\to\gamma_{B,r}\). Then, strongly in \(H^m\),
\[
g_{A,n}:=\sum_{B\in\mathfrak B(A)}\sum_{r<\nu_B}\gamma_{B,r,n}e_{B,r,n}
\longrightarrow
g_A:=\sum_{B\in\mathfrak B(A)}\sum_{r<\nu_B}
(-1)^r\gamma_{B,r}K^{(r)}(\cdot-\lambda_B),
\]
and \(\sum_A\sum_{B\in\mathfrak B(A)}\nu_B=N\le P\).
\end{step2}

\begin{step3}[Limits]
If \(c_{A,n}\to c_A\in\mathbb R\), continuity of translations gives
\(\tau_{c_{A,n}}g_{A,n}\to\tau_{c_A}g_A\) strongly. If
\(|c_{A,n}|\to\infty\), then \(\tau_{c_{A,n}}g_A\rightharpoonup0\), while
\(\|\tau_{c_{A,n}}(g_{A,n}-g_A)\|_{H^m}\to0\). Since the reference centres
of distinct clusters separate, at most one has a finite limit.
\end{step3}
\end{proof}

\subsubsection{Closure and variational consequences}

\begin{proof}[Proof of Theorem~\ref{thm:translated_feature_geometry}]
\begin{step1}[Strong approximation]
Consider a block \(\sum_{r<\nu}w_rK^{(r)}(\cdot-b)\). Choose distinct
\(\lambda_0,\ldots,\lambda_{\nu-1}\) and solve
\(\sum_{j<\nu}a_j^\varepsilon(\varepsilon\lambda_j)^r/r!=(-1)^rw_r\) for
\(0\le r<\nu\). Then \(a^\varepsilon=O(\varepsilon^{-(\nu-1)})\), while
the Banach-valued Peano remainder of \(F\) at \(b\) is
\(o(\varepsilon^{\nu-1})\), so
\(\sum_{j<\nu}a_j^\varepsilon F(b+\varepsilon\lambda_j)
\to\sum_{r<\nu}w_rK^{(r)}(\cdot-b)\) in \(H^m\). Applying this construction
to each block gives an approximation using at most \(P\) translates.
\end{step1}

\begin{step2}[Closure]
Every weakly convergent sequence \((u_n)\subset\mathcal M_P\) is bounded.
By Lemma~\ref{lem:translated_cluster_decomposition}, its limit is the
derivative sum associated with the reference centre that has a finite
limit, or zero if all reference centres escape to infinity. In either case,
the total multiplicity is at most \(P\). Together with the strong
approximation in Step~1, this characterizes both the weak limits and the
strong closure.

To prove weak sequential closedness, let
\(v_n\in\mathcal C_P\) and \(v_n\rightharpoonup v\). Choose
\(u_n\in\mathcal M_P\) with \(\|u_n-v_n\|_{H^m}\le1/n\); then
\(u_n\rightharpoonup v\), so \(v\in\mathcal C_P\).
\end{step2}
\end{proof}

\begin{proof}[Proof of Remark~\ref{rem:translated_arbitrary_sequences}]
\begin{step1}[Strong convergence and representation]
Let \(\alpha>0\) be a modulus of strong convexity for \(J\). The midpoint
inequality gives
\(\tfrac{\alpha}{8}\|u_n-u^\ast\|_{H^m}^2\le\tfrac12(J(u_n)-J(u^\ast))\to0\).
Hence \(u_n\to u^\ast\) strongly. The derivative representation then follows
from Theorem~\ref{thm:translated_feature_geometry}.
\end{step1}

\begin{step2}[Existence of minimizing sequences]
Conversely, if \(u^\ast\in\mathcal C_P\), strong approximation by elements of
\(\mathcal M_P\) and continuity of \(J\) give a sequence with
\(J(u_n)\to J(u^\ast)\). Such a sequence exists if and only if the constrained
infimum equals \(J(u^\ast)\).
\end{step2}
\end{proof}

\begin{proof}[Proof of Corollary~\ref{cor:completed_variational_attainment}]
\begin{step1}[Variational attainment]
By Theorem~\ref{thm:translated_feature_geometry}, \(\mathcal C_P\) is weakly
sequentially closed. Since \(0\in\mathcal C_P\), we can choose a minimizing
sequence, which is bounded by coercivity. Weak compactness in \(H\) and
weak lower semicontinuity then yield a minimizer in the completed class.
\end{step1}

\begin{step2}[Equality of the infima]
The equality of the infima follows from strong continuity and the density
of the original class in the completed class.
\end{step2}
\end{proof}

\subsection{Robust nonattainment}
\label{app:robust_nonattainment}

We write \(K_c=K(\cdot-c)\), denote the parameters of \(u_0\) by
\(\eta_0\), and set \(\eta=(a,b,c,(d_j,c_j)_{j=1}^{P-2})\).
For \(t=h^2>0\), define
\begin{equation}\label{eq:robust_pair_chart}
\Psi(\eta,t)=\frac a2\bigl(K_{c-h}+K_{c+h}\bigr)
+\frac b{2h}\bigl(K_{c-h}-K_{c+h}\bigr)+\sum_{j=1}^{P-2}d_jK_{c_j}.
\end{equation}
The following lemma gives a local parametrization of the completed class
and an expression for the norm of the coefficients at positive separation.

\begin{lemma}[Local parametrization near a simple collision]
\label{lem:robust_local_parametrization}
Under the hypotheses of Theorem~\ref{thm:robust_nonattainment},
\(\Psi\) extends to a jointly \(C^2\) map on \(t\ge0\), one-sided at
\(t=0\), with
\[
\Psi(\eta,0)=aK_c+bK_c'+\sum_jd_jK_{c_j},\qquad
\partial_t\Psi(\eta,0)=\frac a2K_c''+\frac b6K_c''',
\]
and \(D\Psi(\eta_0,0)\) is injective. Every sufficiently small closed box
\(\Omega=\Omega_\eta\times[0,t_*]\), with \(\eta_0\) interior to \(\Omega_\eta\), admits
\(r>0\) such that:
\begin{enumerate}[label=\textup{(\roman*)},leftmargin=*]
\item \emph{Local parametrization.} Every \(v\in\mathcal C_P\) with
\(\|v-u_0\|_{H^m}<r\) equals \(\Psi(\eta,t)\) for some \((\eta,t)\in\Omega\).
\item \emph{Representations.} For \((\eta,t)\in\Omega\) with \(t>0\),
\(\Psi(\eta,t)\) is a sum of exactly \(P\) translates at distinct centres,
with nonzero coefficients. This representation in \(\mathcal M_P\) is
unique up to permutation, and its
weight vector satisfies
\begin{equation}\label{eq:robust_weight_identity}
\|w\|_2^2=S(\eta)+\frac{b^2}{2t},\qquad
S(\eta):=\frac{a^2}{2}+\sum_jd_j^2.
\end{equation}
For \(t=0\), \(\Psi(\eta,0)\notin\mathcal M_P\).
\end{enumerate}
\end{lemma}

Write \(F(s)=K_s\), so that \(F\in C^5(\mathbb R;H^m(\mathbb R))\) for
\(K\in H^{m+5}\), with \(F'=-\tau_sK'\), \(F''=\tau_sK''\) and
\(F'''=-\tau_sK'''\).

\begin{proof}[Proof of Lemma~\ref{lem:robust_local_parametrization}]
\begin{step1}[Smoothness and injectivity]
The functions
\(\tfrac12[F(c-h)+F(c+h)]\) and \([F(c-h)-F(c+h)]/(2h)\) are even and
\(C^4\) in \(h\); for the second, use the Hermite--Genocchi formula
\([F(c-h)-F(c+h)]/(2h)=-\int_0^1F'(c-h+2hs)\,ds\).
An even \(C^4\) function of \(h\) is \(C^2\) in \(t=h^2\ge0\), jointly
in the other parameters. Taylor expansion through \(h^2\) gives the
stated boundary values of \(\Psi\) and \(\partial_t\Psi\).

Since \(\partial_cK_c=-K_c'\), \(b_0\ne0\) and
\(d_{j,0}\ne0\), the range of \(D_\eta\Psi(\eta_0,0)\) is
\[
T=\operatorname{span}\{K_{c_0},K_{c_0}',K_{c_0}'',
K_{c_{j,0}},K_{c_{j,0}}':1\le j\le P-2\}.
\]
By linear independence (Section~\ref{subsec:prelim_translation}),
\(\dim T=2P-1\), so \(D_\eta\Psi(\eta_0,0)\) is injective and
\(K_{c_0}'''\) does not belong to \(T\). Thus
\(q:=(I-\Pi_T)\partial_t\Psi(\eta_0,0)\ne0\), where \(\Pi_T\) is the
orthogonal projection onto \(T\), and \(D\Psi(\eta_0,0)\) is injective.
\end{step1}

\begin{step2}[Choice of the parameter domain]
Choose \(\Omega\) small enough that \(b\) and the \(d_j\) remain
nonzero throughout the domain, the centres \(c,c_j\) remain pairwise separated by some
\(\delta>0\), and \(t_*<\delta^2/4\) with
\(|b|/(2\sqrt{t_*})>|a|/2\). Then, for \(t>0\), the ordinary pair
coefficients \(a/2\pm b/(2\sqrt t)\) are nonzero and the \(P\) centres
\(c\pm\sqrt t\), \(c_j\) are distinct.
\end{step2}

\begin{step3}[Representations]
For \(t>0\), if \(\Psi(\eta,t)\) also equals
\(\sum_{k\le P}w_k'K_{z_k}\), linear independence of translates implies that the
\(P\) translates and their coefficients agree up to permutation. Substituting
the weights \(a/2\pm b/(2\sqrt t)\) and \(d_j\) yields~\eqref{eq:robust_weight_identity}.
At \(t=0\), a representation \(\Psi(\eta,0)=\sum_{k\le P}w_k'K_{z_k}\)
would express \(K_c'\), whose coefficient \(b\) is nonzero, as a linear
combination of finitely many translates, contradicting independence.
\end{step3}

\begin{step4}[Local parametrization]
We first consider a sequence \(v_n\in\mathcal M_P\) converging to \(u_0\).
We show that eventually \(v_n=\Psi(\eta_n,t_n)\), with
\((\eta_n,t_n)\to(\eta_0,0)\). It suffices to establish this on a further
subsequence of every subsequence. After combining terms with the same
centre and discarding those with zero coefficient, apply
Lemma~\ref{lem:translated_cluster_decomposition}.
The profile \(\tau_{c_A}g_A=u_0\) associated with a finite reference centre
must agree with the unique derivative representation of \(u_0\). The
derivative term at \(c_0\) requires at least two centres in its cluster,
and each singleton requires one. Together, these account for all \(P\)
terms: exactly two centres converge to \(c_0\), one to each \(c_{j,0}\),
and none escape or remain unused.

Write the colliding centres as \(c_n\pm h_n\), \(h_n>0\). In terms of their Newton
coordinates \(\gamma_{0,n},\gamma_{1,n}\), the local coordinates are
\(a_n=\gamma_{0,n}\), \(b_n=\gamma_{0,n}h_n-\gamma_{1,n}\).
The cluster decomposition lemma ensures that these coordinates, the
singleton weights and all centres converge. Hence
\(v_n=\Psi(\eta_n,h_n^2)\) with \((\eta_n,h_n^2)\to(\eta_0,0)\).

It follows by contradiction that, for some \(r>0\), every element of the original class
in \(B_{2r}(u_0)\) has coordinates in \(\Omega\). Now approximate any
\(v\in\mathcal C_P\cap B_r(u_0)\) by elements of the original class \(v_n\to v\),
eventually in \(B_{2r}(u_0)\). By compactness of \(\Omega\), the coordinates
converge along a subsequence. Continuity then implies
\(v=\Psi(\eta,t)\). This proves (i) for every sufficiently small parameter box.
\end{step4}
\end{proof}

\begin{proof}[Proof of Theorem~\ref{thm:robust_nonattainment}]
Let \(\Omega\) and \(r\) be as in Lemma~\ref{lem:robust_local_parametrization},
and set \(F(\eta,t):=J(\Psi(\eta,t))\).

\begin{step1}[Uniform convexity in local coordinates]
At \((g,\eta,t)=(u_0,\eta_0,0)\), the residual \(\Psi-g\) vanishes.
The Hessian is therefore \(D\Psi^*D\Psi\), positive definite by injectivity
of \(D\Psi(\eta_0,0)\). After shrinking \(\Omega\) and \(r\), continuity ensures
that the Hessian remains uniformly positive on \(\Omega\) for all \(g\) close to \(u_0\).
\end{step1}

\begin{step2}[An open set of nonattainment]
For the target \(g_\varepsilon=u_0-\varepsilon q\), at \((\eta_0,0)\),
\(D_\eta F=0\) and
\(\partial_tF=\varepsilon\|q\|_H^2>0\). For small
\(\varepsilon>0\), the implicit-function theorem for
\(D_\eta F(\eta,0)=0\), with the target as a parameter, yields a \(C^1\)
map \(\eta_\ast\), with values in the interior of \(\Omega_\eta\), defined on an open ball \(U\) about
\(g_\varepsilon\) small enough
that \(\ell:=\partial_tF(\eta_\ast,0)>0\). Uniform positivity of the
Hessian yields \(\mu>0\) with
\begin{equation}\label{eq:robust_local_energy_bound}
F(\eta,t)-F(\eta_\ast,0)
\ge \ell t+\mu\bigl(\|\eta-\eta_\ast\|_2^2+t^2\bigr)
\qquad ((\eta,t)\in\Omega,\ g\in U).
\end{equation}
By Corollary~\ref{cor:completed_variational_attainment}, \(J\) has a
minimizer \(v\in\mathcal C_P\); comparison with \(u_0\) implies
\(\|v-u_0\|_H\le2\|g-u_0\|_H\), which is less than \(r\) when
\(\varepsilon\) and the radius of \(U\) are sufficiently small.
Lemma~\ref{lem:robust_local_parametrization}\textup{(i)}
and \eqref{eq:robust_local_energy_bound} show that \(u_\ast=\Psi(\eta_\ast,0)\)
is the unique minimizer. It is not in \(\mathcal M_P\) by
Lemma~\ref{lem:robust_local_parametrization}\textup{(ii)}, and
\(J(u_\ast)>0\), since \(g=u_\ast\) would force \(\ell=0\). The equality of
the infima follows from density. The minimizer varies \(C^1\) with the target, and both
\(\varepsilon\) and the radius of \(U\) can be arbitrarily small.
\end{step2}

\begin{step3}[The sharp coefficient bound]
Let \(E(R)\) denote the coefficient-capped infimum in
\eqref{eq:robust_coefficient_budget}.
Fix \(g\in U\), and write
\(\eta_\ast=(a,b,c,(d_j,c_j)_{j=1}^{P-2})\).
For large \(R\), taking \(\eta=\eta_\ast\) and
\(t_R=b^2/[2(R^2-S(\eta_\ast))]\) in~\eqref{eq:robust_weight_identity},
Taylor expansion yields
\(E(R)\le F(\eta_\ast,t_R)-F(\eta_\ast,0)=\ell t_R+O(t_R^2)\), so
\(\limsup_{R\to\infty}R^2E(R)\le\ell b^2/2\).

For the matching lower bound, given \(R_n\to\infty\), choose ordinary
states \(v_n\) with \(\|w_n\|_2\le R_n\)
and excess energy at most \(E(R_n)+o(R_n^{-2})\). The excess energies tend to
zero, so
\(\|v_n-u_0\|_H\le\sqrt{2J(v_n)}+\|g-u_0\|_H\le2\|g-u_0\|_H+o(1)<r\).
By Lemma~\ref{lem:robust_local_parametrization}, \(v_n\) has local coordinates
in the parameter domain, and \eqref{eq:robust_local_energy_bound} implies \(\eta_n\to\eta_\ast\),
\(t_n\to0\). By the uniqueness in
Lemma~\ref{lem:robust_local_parametrization}\textup{(ii)},
\eqref{eq:robust_weight_identity} holds for these coefficients, so
\(R_n^2t_n\ge b_n^2/2\) and
\(\liminf_nR_n^2E(R_n)\ge\liminf_n\ell R_n^2t_n\ge\ell b^2/2\).
Hence \(\lim_{R\to\infty}R^2E(R)=C:=\ell b^2/2>0\).
\end{step3}
\end{proof}

\subsection{Separated centres in general dimension}
\label{subsec:translated_weight_control_appendix}

Uniform separation of the centres yields a bound on the weights in any
dimension. Here \(K\)
satisfies Assumption~\ref{asu:kernel}, and \(\mathcal M_P\) consists of
\(u_\theta=\sum_{k=1}^Pw_k\tau_{b_k}K\), with \(w_k\in\mathbb R\) and
\(b_k\in\mathbb R^d\). Write
\(\delta(b):=\min_{k\ne\ell}|b_k-b_\ell|\), taking \(\delta(b)=+\infty\)
when \(P=1\).

\begin{lemma}[Separated centres and forced collisions]
\label{lem:translated_weight_control}
Let \(d,P\ge1\) be integers.
\begin{enumerate}[label=\textup{(\roman*)},leftmargin=*]
\item For every \(s>0\) there is \(C_s<\infty\) such that
\(\|w\|_2\le C_s\|u_\theta\|_{H^m(\mathbb R^d)}\) whenever
\(\delta(b)\ge s\).
\item Let \(P\ge2\). Then \(\delta(b_n)\to0\) if and only if there are
\(w_n\in\mathbb R^P\) with \(\|w_n\|_2=1\) and
\(\|\sum_kw_{k,n}\tau_{b_{k,n}}K\|_{H^m}\to0\).
\item Let \(P\ge2\). If \(u_n=\sum_kw_{k,n}\tau_{b_{k,n}}K\to u\) strongly
in \(H^m(\mathbb R^d)\) and \(u\notin\mathcal M_P\), then
\(\|w_n\|_2\to\infty\) and \(\delta(b_n)\to0\).
\end{enumerate}
\end{lemma}

\begin{proof}
\begin{step1}[Uniform weight bound]
By Assumption~\ref{asu:kernel}, \(\widehat K\) is continuous and nowhere
zero. Hence
\[
\kappa_m(R):=\inf_{|\xi|\le R}p_m(\xi)|\widehat K(\xi)|^2>0.
\]
Let \(\mu_d\) be the first Dirichlet eigenvalue of \(-\Delta\) on the unit
ball. For \(R>R_0(d,s):=2\sqrt{\mu_d}/s\), the Ingham--Beurling inequality
\cite[Theorem~8.1]{komornik2005fourier} provides \(C_d(R,s)>0\) such that
\[
\int_{B(0,R)}\Bigl|\sum_na_ne^{i\lambda_n\cdot\xi}\Bigr|^2\,d\xi
\ge C_d(R,s)\sum_n|a_n|^2
\]
for finitely supported complex \((a_n)\) and \(s\)-separated
\((\lambda_n)\subset\mathbb R^d\). Plancherel then yields
\[
\|u_\theta\|_{H^m}^2
=\int_{\mathbb R^d}p_m|\widehat K|^2
 \Bigl|\sum_kw_ke^{-ib_k\cdot\xi}\Bigr|^2d\xi
\ge\kappa_m(R)\,C_d(R,s)\,\|w\|_2^2,
\]
so \(C_s=[\kappa_m(R)C_d(R,s)]^{-1/2}\) satisfies (i).
\end{step1}

\begin{step2}[Collision characterization]
Part (i) rules out such
\(w_n\) along any subsequence whose separation is bounded below.
Conversely, place weights \(\pm2^{-1/2}\) on a closest pair and zero
elsewhere. The resulting norm is \(2^{-1/2}\|K-\tau_{y_n}K\|_{H^m}\to0\), where \(y_n\)
is the difference of that pair.
\end{step2}

\begin{step3}[Forced collisions]
Suppose that the weights are bounded along a subsequence. Passing to a
further subsequence, we may assume that the weights converge and each
centre either converges or diverges in norm. The terms with
convergent centres converge strongly to some \(v\in\mathcal M_P\), and the
terms whose centres diverge converge weakly to zero. Hence \(u=v\), a contradiction, so
\(\|w_n\|_2\to\infty\). Since \((u_n)\) is bounded, (i) excludes every
subsequence with \(\delta(b_n)\) bounded below.
\end{step3}
\end{proof}

\section{Conversion to derivative blocks}
\label{app:mixed_completion_proof}

\subsection{Stable coordinates}
\label{app:stable_coordinates}

Linear combinations of translates at distinct centres can be expressed in
terms of normalized divided differences. For example, with
\(F(t)=K(\cdot-t)\), we have
\[
 w_1F(t_1)+w_2F(t_2)=\gamma_0E_0+\gamma_1E_1,
 \qquad E_0=F(t_1),\quad E_1=\frac{F(t_2)-F(t_1)}{t_2-t_1},
\]
where \(\gamma_0=w_1+w_2\) and
\(\gamma_1=(t_2-t_1)w_2\). The higher-order Newton functions are
\(E_r=r![t_1,\ldots,t_{r+1}]F\), defined in
\eqref{eq:completed_newton_coordinates} and extended continuously through
repeated centres, with \(E_r(c,\ldots,c)=(-1)^rK^{(r)}(\cdot-c)\).
The following proposition gives bounds for the Newton coefficients and for
the error in the Taylor replacement of Section~\ref{subsec:completion_algorithm}.
The centres may coalesce at different rates.

\subsection{Uniform stability and transfer bounds}
\label{app:stable_transfer_bounds}

\Needspace{16\baselineskip}
\begin{proposition}[Stable conversion through collisions]
\label{prop:mixed_completion}
Let \(0\ne K\in H^{m+P-1}(\mathbb R)\) be real-valued.

\textup{(i)} \textbf{Stable representation.}
Every \(u\in\mathcal C_P\) admits a representation in the Newton
coordinates~\eqref{eq:completed_newton_coordinates} of the form
\[
 u=\sum_{r=0}^{N-1}\gamma_rE_r(t_1,\ldots,t_{r+1}),
 \qquad N\le P.
\]
Here the centres are real and ordered, and may repeat. When they lie in a
fixed bounded interval, a bound on the Sobolev norm of the represented
function gives a uniform bound on its coefficients.

\textup{(ii)} \textbf{Convergence of the transfer.}
Let \(u_h\) be a family of clusters of multiplicity \(N_h\le P\), with uniformly bounded
\(H^m\) norms and diameters tending to zero as \(h\to0\).
Choose \(\beta_h\) between the smallest and largest centres, and let
\(M_{r,h}\) denote the corresponding Taylor moments at \(\beta_h\). Then
\[
 T_h=\sum_{r=0}^{N_h-1}M_{r,h}K^{(r)}(\cdot-\beta_h),
 \qquad \|u_h-T_h\|_{H^m}\longrightarrow0.
\]
The multiplicity of \(T_h\) is at most \(N_h\), irrespective of the
relative rates at which the centres coalesce.
\end{proposition}

\begin{proof}
\begin{step1}[Uniform bounds for the coefficients]
Set \(H=H^m(\mathbb R)\) and order the nodes
\(t_1\le\cdots\le t_N\), repeating each centre according to its multiplicity,
with \(N\le P\).
The Newton basis depends continuously on the nodes and spans the same
space as the linearly independent Hermite data
(Section~\ref{subsec:prelim_translation}). Its Gram matrix is therefore
positive definite. By compactness, for every \(B>0\),
\begin{equation}\label{eq:completed_synthesis_stability}
 \lambda_{N,B}\|\gamma\|_2^2
 \le\Bigl\|\sum_{r=0}^{N-1}\gamma_rE_r(t_1,\ldots,t_{r+1})\Bigr\|_H^2
 \le\Lambda_{N,B}\|\gamma\|_2^2
 \quad(-B\le t_1\le\cdots\le t_N\le B),
\end{equation}
with \(0<\lambda_{N,B}\le\Lambda_{N,B}<\infty\). Together with
Theorem~\ref{thm:translated_feature_geometry}, this proves (i).
\end{step1}

\begin{step2}[Estimate of the transfer error]
To prove (ii), write
\(u_{\mathcal C}=\sum_i c_iK^{(r_i)}(\cdot-b_i)
=\sum_{r=0}^{N-1}\gamma_rE_r\). Choose \(\beta\) in the convex hull of
the centres and define
\begin{equation}\label{eq:mixed_moment_transfer}
 T_{\mathcal C}=\sum_{s=0}^{N-1}M_sK^{(s)}(\cdot-\beta),
 \qquad M_s=\sum_{i:r_i\le s}
 c_i\frac{(\beta-b_i)^{s-r_i}}{(s-r_i)!}.
\end{equation}
Set \(\rho=\max_j|t_j-\beta|\), which is bounded by the cluster diameter, and
\(\omega_{m,j}(\rho)=\sup_{|a|\le\rho}\|K^{(j)}(\cdot-a)-K^{(j)}\|_H\).
We shall prove that
\begin{equation}\label{eq:cancellation_modulus_bound}
 \|u_{\mathcal C}-T_{\mathcal C}\|_H
 \le\omega_{m,N-1}(\rho)
 \sum_{r=0}^{N-1}|\gamma_r|
       \frac{\rho^{N-1-r}}{(N-1-r)!}.
\end{equation}
Let \(p_\beta(t)=\sum_{s=0}^{N-1}F^{(s)}(\beta)(t-\beta)^s/s!\) and
\(R=F-p_\beta\). Since \(F^{(s)}(t)=(-1)^s\tau_tK^{(s)}\), applying
the linear combination of derivative evaluations that defines the cluster
to \(p_\beta\) gives \(T_{\mathcal C}\). Expressing this linear functional
in Newton form, we obtain
\[
u_{\mathcal C}-T_{\mathcal C}
=\sum_{r=0}^{N-1}\gamma_r\,r![t_1,\ldots,t_{r+1}]R.
\]
Integrating \(F^{(N-1)}(\cdot)-F^{(N-1)}(\beta)\) exactly \(N-1-r\) times yields
\[
\|R^{(r)}(t)\|_H
\le\omega_{m,N-1}(\rho)\frac{|t-\beta|^{N-1-r}}{(N-1-r)!},
\qquad |t-\beta|\le\rho.
\]
The Hermite--Genocchi formula then gives~\eqref{eq:cancellation_modulus_bound}.

By translation invariance, we may take \(\beta=0\) for each cluster.
For sufficiently small diameters, all nodes lie in \([-1,1]\), and
\eqref{eq:completed_synthesis_stability} gives a uniform bound on
\(\gamma\). Continuity of translations implies
\(\omega_{m,N-1}(\rho)\to0\). Since \(1\le N\le P\), the estimate
also applies when the centres and multiplicities vary along the family.
The resulting derivative block has multiplicity at most \(N\).
\end{step2}
\end{proof}

The constants depend on \(N\) and \(B\), and need not remain uniform as
these increase. The estimates concern the coefficients for fixed centres;
they give no corresponding bound on the conditioning of the optimization
problem when the centres are also varied.

\subsection{Numerical evaluation and implementation}
\label{app:stable_transfer_rule}

Direct evaluation of the moments in~\eqref{eq:mixed_moment_transfer} may
involve cancellation between large coefficients. If the cluster is
represented in Newton coordinates, the same moments can be computed from
bounded coefficients by applying its divided-difference functional to the
Taylor polynomial \(p_\beta\) used in the proof.
The error bound in Algorithm~\ref{alg:basic_completion} concerns the
computed replacement and must account for errors in coefficient conversion,
function evaluation and rounding. Estimate~\eqref{eq:cancellation_modulus_bound} bounds the Taylor
remainder; a computable bound on the kernel modulus and on the coefficient
errors is also required. We do not construct such an estimator here.

The additional experiments with stable coordinates use the symmetric basis
\[
 A_h=\tfrac12[K(\cdot-c-he)+K(\cdot-c+he)],\qquad
 B_h=\frac{K(\cdot-c-he)-K(\cdot-c+he)}{2h},\quad h>0,
\]
and represent each pair as \(u_C=aA_h+bB_h\). As \(h\to0\), the basis
functions converge to \(K(\cdot-c)\) and \(-D_eK(\cdot-c)\).
The transfer therefore keeps the coefficients \((a,b)\) in the limiting
basis, as in~\eqref{eq:experimental_pair_transfer}.
For \(d=1\) and direction \(e=+1\), the coefficients in the basis \((K,K')\)
are \((a,-b)\). In the PDE experiments, pairs are prescribed in advance,
and the change in the represented function is estimated by quadrature
(Appendix~\ref{app:numerics_raw_main}). These measurements are not
rigorous error bounds. The experiments do not treat higher-order collisions
or collisions between features that have not been paired in advance.

\Needspace{10\baselineskip}
\section{Numerical protocols}
\label{sec:numerics_details}

This appendix gives the target functions and the numerical procedures
used in Section~\ref{numerics}. The code repository contains the configurations,
results for each initial state, and instructions for reproducing the experiments.
Errors and changes due to transfer are evaluated numerically.

\subsection{The shallow \texorpdfstring{\(\tanh\)}{tanh} experiment}
\label{app:numerics_tanh_settings}

In the bias-collision experiment, we approximate \(u^*\) from
\eqref{eq:numerics_tanh_objective} with slopes fixed at one, initial
weights \((3.5,1.5)\), and initial biases \((1,-1)\). We perform
8,000 Adam updates with learning rate \(10^{-3}\), followed by L-BFGS,
independently in single and double precision. In the slope-collision
experiment, we approximate \(v^*\) in double precision, starting from
the same weights, zero biases, and slopes \((2,0.5)\). We perform
5,000 Adam updates followed by independent L-BFGS refinements with
iteration limits of 200 and 1,000.

\subsection{Gaussian formulas and numerical minimization}
\label{app:manufactured_1d}

The energies and PDE residuals in
Section~\ref{subsec:manufactured_1d} are computed from analytic expressions
for Gaussian inner products.
For each fixed \(\varepsilon\), write \(J\) for the
functional~\eqref{eq:manufactured_1d_energy}.
For the specified perturbation,
\(q\perp\operatorname{span}\{K,K',K''\}\) and
\(\|q\|_{H^1}^2=13\sqrt\pi/160\), so \(u_0=K'\) is stationary in
the completed two-slot family for each target.

The prescribed pair is
\(D_h=[K(\cdot+h)-K(\cdot-h)]/(2h)\), with
\(h=1/(\sqrt2R)\) and ordinary weight norm \(R\).
Its expansion \(D_h=K'+h^2K'''/6+O_{H^1}(h^4)\) gives
\[
\begin{aligned}
 R^4[J(D_h)-J(u_0)]&\longrightarrow
 \frac{15\sqrt\pi}{512}\quad(\varepsilon=0),\\
 R^2[J(D_h)-J(u_0)]
 &\longrightarrow\frac{\varepsilon}{2}\|q\|_{H^1}^2
 \quad(\varepsilon>0).
\end{aligned}
\]
For \(\varepsilon=0\), the stationary symmetric branch \(bD_h\) subject
to the weight bound, with \(h=b/(\sqrt2R)\), has
\(b=1+7/(40R^2)+O(R^{-4})\) and leading state error
\(q/(2R^2)\); its excess-energy constant is therefore
\(13\sqrt\pi/1280\). For positive perturbations, its leading
\(R^{-2}\) constant equals that of the prescribed pair. These constants
determine the normalizations in Figure~\ref{fig:manufactured_1d}.

For each fixed centre and gap, minimization under the weight bound reduces
to a two-dimensional quadratic trust-region problem. We search over a grid
of centres and gaps and refine the resulting candidates, using analytic
Gram products and 80-digit arithmetic in the refinement to resolve small
energy differences. Every pair attaining the lowest energy found is centred
at zero and antisymmetric, with its weight norm attaining the prescribed bound.
The off-centre and separated pairs tested have larger excess energies.
We also check the Gram calculations by independent integration.
This search does not establish global optimality or show that the tested
perturbations lie in the neighbourhood of
Theorem~\ref{thm:robust_nonattainment}.

\subsection{The main comparison and numerical reliability}
\label{app:numerics_raw_main}

In the main PDE comparison, the ordinary weights are fitted using individual
translated kernels with positive gaps. Completed pairs are represented
directly by kernel and directional-derivative columns. Centres are restricted
to \([-2,2]^2\), angles are unrestricted, and finite center half-gaps satisfy
\(0<h\le1.5\), with no positive lower bound. Both methods use the same
solver settings. In fitting the coefficients, we scale the columns and
accept the solution according to its residual, without ridge regularization
or rank truncation. If the coefficient system is ill-conditioned, the run
is stopped. Gradients with respect to the geometric parameters are evaluated
only after the coefficient fit has been accepted.

Table~\ref{tab:raw_training_quadrature} reports the results of repeating
the optimization at higher quadrature orders.

\begin{table}[H]\centering\small
\caption{Completion retains lower median errors with finer training quadrature.}
\label{tab:raw_training_quadrature}
\begin{tabular}{@{}llcccc@{}}\toprule
 & & \multicolumn{2}{c}{Standard optimization} & \multicolumn{2}{c}{Completion}\\
\cmidrule(lr){3-4}\cmidrule(l){5-6}
Problem & Order & Median \(H^1\) & \(\le10^{-7}\) & Median \(H^1\) & \(\le10^{-7}\)\\\midrule
Poisson & 32 & \(1.8\times10^{-5}\) & 0/10 & \(8.1\times10^{-8}\) & 7/10\\
 & 64 & \(1.9\times10^{-5}\) & 0/10 & \(7.9\times10^{-8}\) & 6/10\\
\midrule
Mat\'ern & 64 & \(1.7\times10^{-5}\) & 0/10 & \(2.3\times10^{-7}\) & 5/10\\
 & 96 & \(1.8\times10^{-5}\) & 0/10 & \(1.5\times10^{-8}\) & 6/10\\
\bottomrule\end{tabular}
\end{table}

Completion uses~\eqref{eq:experimental_pair_transfer}. The conditions for
transfer are checked every 40 accepted iterates and when the optimizer
returns: the full gap must be at most \(0.05\) for Mat\'ern or \(0.12\)
for Poisson, and the absolute \(H^1\) change must be at most \(0.03\),
measured with 64 quadrature points per coordinate. Candidate pairs are
checked again after each transfer and coefficient refit. Each run retains its last
accepted fitted state, including when a trial evaluation fails. Normal
termination of the optimizer does not establish stationarity or global
optimality.

An independent evaluation of the features reproduces the main tables and
figures. For representative finite and completed configurations, comparison
with central differences gives relative gradient discrepancies below
\(1.2\times10^{-8}\) for Mat\'ern and \(1.0\times10^{-8}\) for Poisson.
At each quadrature order used for training, the standard and completion
iterates agree exactly until the first transfer. All measured changes due
to transfer satisfy the tolerance; across both training orders, they are
below \(1.01\times10^{-3}\) for Mat\'ern and \(3.00\times10^{-2}\) for
Poisson. When the returned states are evaluated again with 128 points per
coordinate, their errors remain unchanged at the precision displayed in
the main tables of final approximation errors.
Completion still gives lower
median errors, although no transfer occurs for Mat\'ern initial state~1
at order 96, and both methods return an error of \(4.8\times10^{-4}\).
All ten initial states are included; errors are evaluated independently
with 128 points per coordinate, without further optimization or coefficient
refitting.
For the Mat\'ern problem, refitting the coefficients with the geometry fixed
at that of the target gives relative \(H^1\) errors of about
\(2.3\times10^{-8}\) at training order 64 and \(5.3\times10^{-10}\) at
order 96. This illustrates the effect of training quadrature near the target.

\subsection{The Poisson--Neumann target and initial states}
\label{app:numerics_poisson}

For \(K(y)=e^{-|y|^2/(2s^2)}\), \(s=0.45\), and unit directions
\(e_k=(\cos\theta_k,\sin\theta_k)\), the target is
\[
 u^\ast=\sum_{k=1}^{4}
 \bigl(a_kK(\cdot-c_k)-b_kD_{e_k}K(\cdot-c_k)\bigr)
 +\sum_{j=1}^{2}a_{0j}K(\cdot-c_{0j}),
\]
with
\[
\begin{aligned}
 (c_1,\theta_1,a_1,b_1)&=((-0.55,-0.42),0.34,1.1,-1.3),\\
 (c_2,\theta_2,a_2,b_2)&=((0.47,-0.40),1.26,-0.8,1.5),\\
 (c_3,\theta_3,a_3,b_3)&=((-0.44,0.54),2.04,1.4,0.9),\\
 (c_4,\theta_4,a_4,b_4)&=((0.57,0.43),2.86,-1.2,-1.1).
\end{aligned}
\]
The singletons \((c_{0j},a_{0j})\) are \(((-0.03,0.62),0.9)\) and
\(((0.05,-0.56),-0.7)\). The mean in~\eqref{eq:poisson_energy} is
\(M=\int_\Omega u^\ast\approx-1.16643\).
For initial states 10--19, we perturb the reference pair centres
\((-0.52,-0.46)\), \((0.50,-0.44)\), \((-0.48,0.50)\), \((0.54,0.47)\)
and singleton centres \((0,0.58)\), \((0.02,-0.60)\) by uniform coordinate
offsets in \([-0.10,0.10]\). We perturb the angles \(0.30,1.20,2.10,2.80\)
by offsets in \([-0.18,0.18]\), and sample the center half-gaps uniformly from
\([0.02,0.05]\). The configuration was selected using initial states 0--4
and fixed before initial states 10--19 were sampled.

\Needspace{12\baselineskip}
\subsection{The Mat\'ern target and initial states}
\label{app:numerics_matern}

With \(\Phi\) from~\eqref{eq:matern_kernel} and
\(e_k=(\cos\theta_k,\sin\theta_k)\), the target is
\[
 u^\ast=\sum_{k=1}^{2}
 \bigl(a_k\Phi(\cdot-c_k)-b_kD_{e_k}\Phi(\cdot-c_k)\bigr)
 +a_0\Phi(\cdot-c_0),
\]
where
\[
\begin{aligned}
 (c_1,\theta_1,a_1,b_1)&=((-0.32,-0.27),0.40,1.2,1.5),\\
 (c_2,\theta_2,a_2,b_2)&=((0.33,0.27),1.75,-0.9,1.1),\\
 (c_0,a_0)&=((0,0.05),0.6).
\end{aligned}
\]
The loads are those of~\eqref{eq:matern_energy}. For initial states 0--9,
we perturb the reference pair centres \((-0.3,-0.25)\), \((0.3,0.25)\)
by independent uniform coordinate offsets in \([-0.12,0.12]\), and the
angles \(0.35,1.65\) by offsets in \([-0.2,0.2]\). The singleton centre
is sampled uniformly from \([-0.12,0.12]^2\), and the center half-gaps are sampled
uniformly from \([0.02,0.05]\).

\subsection{Evaluation accuracy and the choice of coordinates}
\label{subsec:precision_effects}
\label{app:numerics_precision}
\label{app:stable_matern_controls}
\label{app:stable_poisson_controls}

We use prescribed Gaussian difference quotients for the first and third
derivatives to study evaluation accuracy independently of optimization.
Compared with a 100-digit reference, stable evaluation gives relative
\(H^1(-12,12)\) errors that level off near \(0.3\,\varepsilon_{\rm mach}\)
as the gaps shrink, in both single and double precision. Direct summation
loses accuracy in this limit.

Additional optimization experiments use stable pair coordinates throughout,
with four variants: retaining finite pairs, restarting without transfer,
optimizing \(h^2\), and completing eligible pairs. In exact arithmetic,
the finite-gap objectives after coefficient refitting agree with those
in ordinary coordinates; differences arise from numerical arithmetic and
the parametrization of the geometry. At the main training orders,
the stable finite-pair and completed methods have median errors
\(5.3\times10^{-8}\) and \(3.5\times10^{-8}\) for Mat\'ern, and
\(9.2\times10^{-7}\) and \(1.2\times10^{-7}\) for Poisson.
Using the squared gap as a variable gives \(3.6\times10^{-8}\) and
\(1.4\times10^{-7}\), respectively. The relative performance depends on
the problem and quadrature order: for Mat\'ern at training order 96, the
median error with the squared-gap variable is \(1.3\times10^{-8}\),
compared with \(2.1\times10^{-7}\) for stable completion. These results
illustrate how coordinates and quadrature affect optimization. Full results
are available in the repository.

\FloatBarrier
\end{document}